\documentclass[11pt]{article}

\RequirePackage{my_packages} 
\RequirePackage{my_theorems}
\RequirePackage{my_macros}
\RequirePackage{my_tikz}

\usepackage[normalem]{ulem}

\usepackage[backend = biber,        
            language = english ,
            style = numeric-comp ,   
            sorting = nyt,
            sortcites = true,
            firstinits = true,
            isbn = false,
            url = false,
            doi = false,
            maxnames = 6,
            backref=true
            ]{biblatex}

\DeclareNameAlias{sortname}{last-first}
\renewbibmacro{in:}{%
  \ifentrytype{article}{}{%
  \printtext{\bibstring{in}\intitlepunct}}}
\AtEveryBibitem{
 \clearlist{address}
 \clearfield{date}
 \clearfield{eprint}
 \clearfield{isbn}
 \clearfield{issn}
 \clearlist{location}
 \clearfield{month}
 \clearfield{series}
 \clearlist{language}
 \clearfield{note}
 
}

\DeclareNameAlias{labelname}{given-family}

\title{
Additive subgroups of $\QQ$: logarithmic description and intersection configurations
}

\author{Jordi Delgado, Antoni Massegú and Enric Ventura}
\date{\vspace{-7pt}
    Departament de Matemàtiques\\
    Universitat Politècnica de Catalunya\\[17pt]
    \today
}

\newcommand{\Addresses}{{
  \bigskip
  \footnotesize

  Jordi Delgado\\\nopagebreak
  \textsc{Departament de Matemàtiques\\Universitat Politècnica de Catalunya, Spain}\\\nopagebreak 
  \url{jorge.delgado@upc.edu}

  \medskip

  Enric Ventura\\\nopagebreak 
  \textsc{Departament de Matemàtiques \& ImTech \\ Universitat Politècnica de Catalunya, Spain}\\\nopagebreak 
  \url{enric.ventura@upc.edu }

   \medskip

  Antoni Massegú\\\nopagebreak 
  \textsc{Departament d'Educació Generalitat de Catalunya\\Universitat Politècnica de Catalunya, Spain}\\\nopagebreak 
  \url{antoni.massegu@upc.edu }

}}

\begin{document}

\maketitle

\begin{abstract}
We develop a logarithmic framework to study the lattice of subgroups of the additive group of rational numbers. By encoding positive rationals via their prime exponent sequences, one obtains a bijection between $\QQ^+$ and finitely supported integral sequences indexed by the prime numbers. This correspondence extends to arbitrary subgroups of $(\QQ,+)$ through a logarithmic greatest common divisor, yielding a classification of subgroups in terms of eventually nonpositive sequences in $\ZZ\cup\{-\infty\}$. 

Within this framework, subgroup inclusion, sum, product, and intersection admit simple coordinatewise descriptions, providing a transparent interpretation of the subgroup lattice and recovering several classical results such as the subgroup classification up to isomorphism.

Exploiting this perspective, we obtain a complete characterization of the intersection configurations realizable in subgroups of $(\QQ,+)$. For rank configurations, realizability is characterized by the decreasing condition together with restrictions on the possible ranks and on the minimal zero sets, and, in the finite-support case, by an additional cardinality condition. For binary configurations, this reduces to saying that every decreasing configuration is realizable in the infinite-support case, whereas in the finite-support case the mentioned cardinality condition is required.
\end{abstract}






\vspace{10pt}




\section*{Introduction}
The additive group of rational numbers $(\QQ,+)$ is one of the most classical examples of a torsion-free, non-(finitely generated) abelian group. Every finitely generated subgroup is cyclic, but arbitrary subgroups may
encode independent divisibility phenomena at infinitely many primes. Beaumont and Zuckerman gave a complete parametrization of the additive subgroups of $\QQ$  in~\cite{beaumontCharacterizationSubgroupsAdditive1951a}, using a positive integer together with a sequence of bounds on the powers of primes that may occur in denominators.

The purpose of this paper is twofold. First, we give an intrinsic logarithmic reformulation of this classical parametrization and use it to provide a unified description of the algebraic and order-theoretic structure of the set of subgroups of $(\QQ,+)$. Second, we apply this framework to determine the rank and finite generation patterns that can occur among intersections of finite families of subgroups. The latter problem is motivated by the
theory of intersection configurations recently developed
in~\cite{delgadoIntersectionConfigurationsFree2024}.

Let $\PP$ be the set of prime numbers. Our approach is based on encoding positive rational numbers via their prime exponents (also called \emph{valuations}
in number theoretical contexts),
see \cite{atiyah_introduction_1994,
fuchs_infinite_1970,
serre_course_1973,}. More precisely, every $r\in\QQ^+$ can be uniquely written as a finite product $r=\prod_{p\in\PP} p^{e_p}$ with $e_p\in\ZZ$, and hence identified with a finitely supported sequence $(e_p)_{p\in\PP}\in\Seq_0(\ZZ)$, called its logarithm, $\plog(r)$. 
The resulting bijection
$\QQ^+ \to \Seq_0(\ZZ)$
transforms multiplication of rational numbers into componentwise addition of finitely supported sequences.


We extend this correspondence from elements to subgroups by introducing the \emph{logarithmic greatest common divisor} of an arbitrary nonempty subset $R\subseteq\QQ^+$, defined as the coordinatewise infimum of the sequences in $\plog(R)$. This invariant, denoted by $\lgcd(R)$, is a sequence indexed by the prime
numbers with entries in $\Zb = \ZZ\cup\set{-\infty}$, and encodes precisely the subgroup of $(\QQ,+)$ generated by $R$. In particular, it leads to a bijective correspondence between the set $\Sgp^*(\QQ,+)$ of nontrivial subgroups of $(\QQ,+)$ (including the non finitely generated ones) and the set $\Seq_{\leq 0}(\Zb)$ of eventually nonpositive sequences in $\Zb$.

\begin{thm*}
The map
\[
\begin{array}{rcl}
\Sgp^*(\QQ,+) &\to &\Seq_{\leq 0}(\Zb)\\
H &\mapsto &\lgcd(H)
\end{array}
\]
is a bijection and an anti-isomorphism of partially ordered sets.
\end{thm*}

Within this framework, the structure of the lattice of subgroups of $(\QQ,+)$ becomes remarkably transparent. Under the correspondence $H\mapsto\lgcd(H)$, subgroup inclusion corresponds to reverse coordinatewise order, while subgroup sum and intersection correspond, respectively, to coordinatewise minimum and maximum.
Thus, the subgroup lattice of $(\QQ,+)$ is anti-isomorphic to the partially ordered set $\Seq_{\leq 0}(\Zb)$ of eventually nonpositive sequences in $\Zb$.
This provides a unified logarithmic reformulation of several classical facts concerning subgroups of $\QQ$, including the cyclicity of finitely generated subgroups, the description of subgroup indices, and the classification of subgroups up to isomorphism or commensurability.

Besides providing a unified and self-contained account of all these facts, the logarithmic language makes new questions concerning families of intersections accessible. Motivated by the theory of intersection configurations
developed in~\cite{delgadoIntersectionConfigurationsFree2024}, we study the rank and finite-generation patterns that may occur among intersections of finite families of subgroups of $(\QQ,+)$.
Given subgroups $H_1,\ldots,H_k$ of a fixed subgroup $G\leqslant\QQ$, we record, for every nonempty set of indices $I \subseteq \{1,\ldots,k\}$, the rank of the intersection $H_I=\bigcap_{i\in I}H_i$. This produces a numerical pattern on the nonempty subsets of $\{1,\ldots,k\}$ which becomes boolean when we collapse all the finite ranks.

For rank configurations, every realizable pattern is decreasing, takes values only in $\{0,1,\infty\}$, and has the property that every minimal set on which it takes the value $0$ is a singleton. We also prove the following precise converse (see~\Cref{sec: intersections} for the precise definitions and notations).

\begin{thm*}
Let $G\leqslant\QQ$ be a nontrivial subgroup, and let $\chi$ be a rank
intersection configuration. Then $\chi$ is realizable in $G$ if and only if:
\begin{enumerate}[(i)]
\item $\chi$ is decreasing and takes values in $\{0,1,\infty\}$;
\item every minimal set on which $\chi$ takes the value $0$ is a singleton;
\item if $G$ has finite support, then
$|\mathcal M_\infty(\chi)| \leq |\supp_{-\infty}(\lgcd(G))|$.
\end{enumerate}
\end{thm*}

In particular, a binary configuration is realizable if and only if it is decreasing and, when the ambient subgroup has finite support, satisfies the same cardinality constraint.
\bigskip

These results illustrate the main advantage of the logarithmic description: arithmetic restrictions on rational subgroups become
combinatorial conditions on integer-valued sequences, while intersections are reduced to coordinatewise maxima. In this way, the framework both clarifies the classical structure of the lattice of additive subgroups of $\QQ$ and provides an effective tool for studying new intersection phenomena.


\section{Numbers and sequences}

Throughout the paper, $\PP$, $\NN$, $\ZZ$ and $\QQ$ denote the set of prime numbers,
the set of natural numbers (including $0$), the set of integers, and the set of rational numbers, respectively. Several variations of these numeric sets are also considered: 
 \begin{itemize}
\item for $k\in \NN$ we denote by $\NN_{\geq k}=\set{n \in \NN \st n\geq k}$, and we write $[k]=\set{1,\ldots,k}$;
\item $\NN^+ =\NN_{\geq 1}$ denotes the set of positive natural numbers;
\item $\QQ^+ =\set{r\in \QQ \st r>0}$ denotes  the set of positive rational numbers;
\item if $S\subseteq \QQ$, we will write $S^{+}=S\cap \QQ^+$.
\end{itemize}

We will also consider the ordered set $\Zb=\ZZ\cup\{-\infty\}$, where $-\infty<z$ for all $z\in\ZZ$. Addition is extended by declaring $-\infty$ absorbing, that is,
 $$
z+(-\infty)=(-\infty)+z=-\infty \qquad \text{for all } z\in\Zb.
 $$
Whenever coordinatewise differences of elements of $\Zb$ occur, we use the convention
 $$
(-\infty)-(-\infty)=0 
\quad\text{and}
-(-\infty) = +\infty.
 $$
Then, for $a,b\in \Zb$, we define the \defin{distance between} $a$ and $b$ by
 $$
\dist(a,b)=\begin{cases} +\infty & \text{if exactly one of }a,b\text{ is }-\infty, \\ |a-b| & \text{otherwise.} \end{cases}
 $$
Thus $\dist(a,b)\in\NN\cup\{+\infty\}$.

\subsection{Sequences}

Sequences play a fundamental role in this paper. As we will see, it will be natural to consider them indexed by (the countable set) $\PP$. Given a set $A$, we write
 $$
\Seq(A)=\{\,\s{s}\mid \s{s}\colon \PP\to A\,\}
 $$
for the set of all sequences indexed by~$\PP$ and with values in~$A$. If $a\in Im(\s{s})$, we abuse notation and write $a\in \s{s}$. Sometimes, we will also refer to a sequence $\s{s}$ with the notation $\s{s}=(\s{s}(p))_{p\in \PP}$. The sequences considered in this paper will all be numeric; typically we will take $A$ to be $\NN$, $\ZZ$ or $\Zb$. 

If $P\subseteq \PP$, we denote by $\mathbf{1}_P\colon \PP \to \{0,1\}$ the indicator sequence of $P$, that is,
 \begin{equation}\label{eq: indicator}
\mathbf{1}_P(q)=\begin{cases} 1,& q\in P, \\ 0,& q\notin P. \end{cases}
 \end{equation}
In particular, for $p\in\PP$, $\mathbf{1}_p :=\mathbf{1}_{\{p\}}$ is the sequence with entry $1$ at position $p$ and entry $0$ elsewhere. For example,
$\mathbf{1}_5=(0,0,1,0,0,\ldots)$.

\medskip
\noindent
\textbf{Support and type.} For a numeric sequence $\s{s}\in \Seq(A)$, we define its \defin{support} as
 $$
\supp(\s{s})=\set{\,p\in\PP \mid \s{s}(p)\neq 0\,}.
 $$
We say that $\s{s}$ has \emph{finite support} if $|\supp(\s{s})|<\infty$. More generally, if $B \subseteq A$, we define the \defin{$B$-support} of $\s{s}\in \Seq(A)$ as
 $$
\supp_{B}(\s{s})=\set{\,p\in\PP \mid \s{s}(p) \in B\,}=\s{s} \preim(B).
 $$
Some relevant instances of this notation are:
\begin{enumerate}[(a)]
\item the \defin{positive support} of $\s{s}$, $\supp_{>0} (\s{s})=\{\,{p\in\PP} \mid \s{s}(p)>0\,\}$, 
\item the \defin{nonpositive support} of $\s{s}$, $\supp_{\leq 0} (\s{s})=\{\,{p\in\PP} \mid \s{s}(p)\leq 0\,\}$,
\item the \defin{$(-\infty)$-support} of $\s{s}$, $\supp_{-\infty}(\s{s})=\{\,{p\in\PP} \mid \s{s}(p)=-\infty\,\}$.
\end{enumerate}
Note that $\supp_{(-\infty)}(\s{s}) \subseteq \supp_{\leq 0}(\s{s})$ and $\PP =\supp_{>0}(\s{s}) \sqcup \supp_{\leq 0}(\s{s})$.

\medskip
A sequence $\s{s}\in\Seq(\Zb)$ is called \defin{positive} (\resp \defin{negative}, \defin{nonpositive}, \defin{nonnegative}) if $\s{s}(p)>0$ (\resp $\s{s}(p)<0$, $\s{s}(p)\leq 0$, $\s{s}(p)\geq 0$) for all $p\in\PP$. We prepend the adjective \emph{eventually} to any of these terms if the corresponding condition holds for all but finitely many primes $p\in\PP$.

More generally, given a subset $B \subseteq A$ we will denote by $\Seq_B(A)$ the set of sequences with all but finitely many entries in $B$, that is,
 \[
\Seq_B(A) \,=\, \set{s\in \Seq(A) \st s^{-1}(A \setmin B) \text{ is finite} }.
 \]
Accordingly, the set of sequences in $A$ with finite support is denoted by $\Seq_0(A)$, and the set of sequences with finite positive support (\ie the eventually nonpositive ones) is denoted by $\Seq_{\leq 0}(A)$. We say that a sequence $\s{s}\in\Seq(\Zb)$ is of \defin{finite type}
if it has both finite support and integral image (that is, if
$\s{s}\in\Seq_0(\ZZ)$), and of \defin{infinite type} otherwise. The latter occurs in the following (non-exclusive) situations:
\begin{itemize}
\item $\s{s}$ has infinite support (\ie $|\supp(\s{s})|=\infty$); we say that $\s{s}$ is of \defin{horizontal type}; 
\item at least one entry in $\s{s}$ is $-\infty$ (\ie $\supp_{-\infty}(\s{s})\neq \varnothing$); in this case, we say that $\s{s}$ is of \defin{vertical type}.
\end{itemize}

The \defin{distance} between two sequences $\s{s},\s{r}\in\Seq(\Zb)$ is defined by
 \begin{equation}
\dist(\s{s},\s{r})=\sum_{p\in\PP} \dist(\s{s}(p),\s{r}(p)) \in \NN\cup\{+\infty\}.
 \end{equation}
Finally, two sequences $\s{s},\s{r}$ are said to have the same \defin{vertical tail} if $\supp_{-\infty}(\s{s}) = \supp_{-\infty}(\s{r})$, and the same \defin{horizontal tail} if they differ at only finitely many primes. Note that $\dist(\s{s},\s{r})<+\infty$ if and only if $\s{s}$ and $\s{r}$ have the same horizontal and vertical tails.

\medskip

\noindent
\textbf{Order and operations on sequences.} Sequences on $A$ naturally inherit the algebraic (\resp order) structure from $A$, componentwise. More precisely, given $\s{s},\s{r}\in \Seq(A)$, we define:
 \begin{align*}
\s{s}\preceq \s{r} \ &\Leftrightarrow\ \s{s}(p)\leq \s{r}(p)\ \text{ for all } p\in\PP, \\ \s{s} + \s{r} &\,=\, (\s{s}(p) + \s{r}(p))_{p \in \PP}.
 \end{align*}
Accordingly, the maximum, minimum, supremum and infimum of sequences are also defined componentwise (whenever the relevant coordinatewise extrema exist). If $\mathcal{S}=\set{\s{s_i} \st i \in I} \subseteq \Seq(A)$, then 
 \begin{align*}
\sup(\mathcal{S}) \,=\, \Big( \sup_{i\in I}(\s{s_i}(p)) \Big)_{p \in \PP} \quad\text{and}\quad \inf(\mathcal{S}) \,=\, \Big( \inf_{i\in I}(\s{s_i}(p)) \Big)_{p \in \PP} \,.
 \end{align*}

\medskip
\noindent
\textbf{Successors.} Given $\s{s}\in\Seq(\Zb)$, we denote 
\begin{itemize}
    \item the \defin{set of successors} of $\s{s}$ by
 \begin{align*}
\Succ(\s{s})=\{\,\s{x}\in\Seq(\Zb)\mid \s{s}\preceq \s{x}\,\},
 \end{align*}
 \item the \defin{set of successors of $\s{s}$ with finite support} by 
 \[
 \oSucc \s{s}
    = \Succ \s{s} \cap \Seq_0(\Zb)
    = \set{\s{x} \in \Seq_0(\Zb) \st \s{s} \preceq \s{x}},
\]
 \item the \defin{set of successors of $\s{s}$ of finite type} by
 \[
\fSucc \s{s} 
    = \Succ \s{s} \cap \Seq_0(\ZZ)
    = \set{\s{x} \in \Seq_0(\ZZ) \st \s{s} \preceq \s{x}}.
 \]
\end{itemize}

\medskip
\noindent

\section{$\PP$-logarithms}

By the Fundamental Theorem of Arithmetic, every positive integer $n\geq 1$ admits a prime factorization, unique up to reordering,   
 \begin{equation}\label{eq: prime desc N}
n=\prod_{p\in \PP} p^{e_p},
 \end{equation}
where $e_p \in \NN$, and all but finitely many $e_p$'s are equal to~$0$. That is, the sequence of exponents $\s{e} = (e_p)_{p\in \PP}$ in \eqref{eq: prime desc N} is a sequence in $\Seq_0(\NN)$ uniquely associated to $n$; we call it the $\PP$-logarithm of $n$, and we denote it by $\plog(n)$.

\begin{defn}
Let $n=\prod_{p\in \PP} p^{e_p} \in \NN^+$, and let $p \in \PP$. Then, 
 \begin{enumerate}[(i)]
\item the \defin{$p$-logarithm} of $n$ is $\plog_p(n)=e_p \in \NN$;
\item the \defin{$\PP$-logarithm} of $n$ is $\plog(n)=(e_p)_{p\in \PP}\in \Seq_0(\NN)$.
\end{enumerate}
\end{defn}

 
In number theoretical contexts, $\plog_p(n)$ is usually called the \emph{$p$-valuation} of $n$. It is clear that the map $n\mapsto \plog(n)$ is a bijection between $\NN^+$ and~$\Seq_0(\NN)$ with inverse sending every sequence $\s{e} = (e_p)_{p\in \PP} \in \Seq_0(\NN)$ back to the positive natural number $\prod_{p \in \PP} p^{e_p}$, which we abbreviate as $\PP^{\,\s{e}}$. That is, we have a bijection:
 \begin{equation} \label{eq: plog N}
\begin{array}{rcl}
\plog \colon  \NN^+ & \to & \Seq_0(\NN) \\ n & \mapsto & \plog(n) \\ \PP^{\,\s{e}} & \mapsfrom & \s{e}.
\end{array}
 \end{equation}

Since the support of every sequence $\s{e} \in\Seq_0(\NN)$ is finite, we write $(0)_{p \in \PP}=\s{0}$, and we omit the trailing sequence of zeroes in every $\s{e} \neq \s{0}$. For example, we write $\plog(1)=\s{0}$ (\resp $\PP^{\,\s{0}}=1$) and $\plog(99)=\plog(3^2\cdot 11^1)=(0,2,0,0,1)$ (\resp $\PP^{\,(0,2,0,0,1)}=3^2\cdot 11^1=99$). So, $\PP$-logarithms of positive natural numbers can be viewed as vectors with natural entries of non-prefixed finite length (technically, of infinite length with zeroes from some coordinate on). 

In the proposition below we summarize some elementary but relevant properties of the $\PP$-logarithm map \eqref{eq: plog N} on positive integers. We assume the standard multiplication and divisibility relation
in $\NN$, and the component-wise sum
and order ($\preceq$)
in $\Seq_0(\NN)$.


\begin{prop}\label{prop: plog(N)} 
For every $n,m \in \NN^+$, and every $k \in \NN$,
 \begin{enumerate}[(i)]
\item if $n\pm m\neq 0$, then $\plog(|n\pm m|)\succeq \min\{ \plog(n),\, \plog(m)\}$;
\item \label{item: plog(mn) = plog(m) + plog(n)} $\plog(m n)=\plog(m)+\plog(n)$;
\item $\plog(n^k)=k\plog(n)$;
\item \label{item: order} $\plog(m) \preceq \plog(n)$ if and only if $m \divides n$. 
\end{enumerate}
In particular, the bijective map~\eqref{eq: plog N}
is an isomorphism of partially ordered monoids between ${(\NN^+,\cdot,\divides)}$ and $(\Seq_0(\NN),+,\preceq)$. \qed
\end{prop}



This isomorphism allows us to smoothly translate divisibility notions in $\NN^+$ to the language of integer sequences (or $\PP$-logarithms). For example, the standard notions of \emph{greatest common divisor} ($\gcd$) and \emph{least common multiple} ($\lcm$), take a particularly neat logarithmic form.

\begin{lem}\label{lem: gcd N}
Let $S$ be a nonempty subset of $\NN^+$. Then,
 $$
\plog(\gcd(S))=\min(\plog(S))\in \Seq_0(\NN).    
 $$
In particular, $S$ is a coprime family if and only if $\min(\plog(S)) = \s{0}$. Moreover, if $S$ is finite then $\plog(\lcm(S)) = \max(\plog(S)) \in \Seq_0(\NN)$. \qed
\end{lem}

    

It turns out that the $\PP$-logarithm not only smoothly encodes the multiplicative monoid of positive integers but also contains relevant information about the \emph{additive} group of integers. This connection follows from the translation of Bezout's identity to the language of $\PP$-logarithms, and is summarized below.

Let $(\ZZ,+)$ be the additive group of integers. For every subset $S\subseteq \ZZ$, let us denote by $\gen{S}$ the additive subgroup of $\ZZ$ generated by $S$,
and write $\gen{S}^{+}=\gen{S}\cap \NN^+$.


\begin{prop}\label{prop: Bezout N}
Let $\emptyset \neq S\subseteq \NN^+$, and let $x\in \NN^+$. Then,
 \begin{equation} \label{eq: Bezout}
x\in \gen{S} \ \Longleftrightarrow \ \plog(x) \succeq \min(\plog(S)) \,.
 \end{equation}
Equivalently, $\gen{S} =\pm \PP^{\,\oSucc \min(\plog S)}\cup \{0\}.$
\end{prop}

\begin{proof}
From Bezout's identity, it is well-known that $\gen{S}=\gen{\gcd(S)}$. And from the above discussion, 
 \begin{align*}
x \in \gen{S}^+ &\,\Leftrightarrow\, x \in \gen{\gcd(S)}^+ \\ &\,\Leftrightarrow\, \gcd(S) \divides x \\ &\,\Leftrightarrow\, \plog(\gcd(S)) \preceq \plog(x) \\
&\,\Leftrightarrow\, \min(\plog(S)) \preceq \plog(x). 
 \end{align*}
The equality $\gen{S}=\pm \PP^{\,\oSucc \min(\plog S)}\cup \{0\}$ follows. 
\end{proof}


\bigskip
Our next goal is to extend the $\PP$-logarithm function to the set of positive rational numbers $\QQ^+ = \set{r\in \QQ \st r>0}$, and to show how the logarithmic approach to subgroups of $(\ZZ,+)$ can be extended to obtain a neat description of subgroups of the additive group $(\QQ,+)$. For $R\subseteq \QQ$, we denote by $\gen{R}$ the additive subgroup of~$\QQ$ generated by $R$, and we write $R^{+}=R\cap \QQ^+$.


\begin{defn}
The \defin{$\PP$-logarithm} of a positive rational number $a/b\in \QQ^{+}$ is
 $$
\plog(a/b)=\plog(a)-\plog(b)\in \Seq_0(\ZZ).
 $$
For every $p\in \PP$, we define $\plog_p =\pi_p \circ \plog$, where $\pi_p$ is the projection to the $p$-th coordinate on $\Seq_0(\ZZ)$. For $r\in \QQ^+$, $\plog_p(r)\in \ZZ$ is also known as the \emph{$p$-valuation} of the rational number $r$. We have $\plog(r) = (\plog_p(r))_{p\in\PP}$.
\end{defn}

It is clear that this is well  defined since, by \Cref{prop: plog(N)}.\ref{item: plog(mn) = plog(m) + plog(n)}, for every $c \in \NN^{+}$, $\plog(ac/bc)=\plog(a/b)$. In other words, every positive rational $a/b\in \QQ^+$ admits a prime factorization, unique up to reordering,   
 \begin{equation}\label{eq: prime desc Z}
\frac{a}{b}=\prod_{p\in \PP} p^{e_p},
 \end{equation}
where $e_p \in \ZZ$, and all but finitely many $e_p$'s are equal to~$0$; and the sequence of exponents $\s{e}=(e_p)_{p\in \PP}$ in~\eqref{eq: prime desc Z} is the $\PP$-logarithm of $a/b$, $\plog (a/b)$, a sequence in $\Seq_0(\ZZ)$ uniquely associated to $a/b$.

Since the $\PP$-logarithm of any positive rational number is a sequence of integers with finite support, as we did for positive integers, we omit the trailing zeroes in $\plog(a/b)$; for example, we write $\plog(25/27)=(0,-3,2)$.    
 
\begin{prop}\label{prop: plog(Q)}
The map
 \begin{equation} \label{eq: plog(Q)}
\begin{array}{rcl}
\plog \colon \QQ^{+} & \to & \Seq_0(\ZZ)  \\
a/b & \mapsto & \plog{a/b}
\end{array}
 \end{equation}
is 
a bijection 
with inverse $\s{e} \mapsto \PP^{\,\s{e}} = \prod_{p\in \PP} p^{e_p}$. Moreover, for every  
$r,s \in \QQ^+$ and every $k\in \ZZ$
 \begin{enumerate}[(i)]
\item\label{item: logsuma} if $r\pm s\neq 0$, then $\plog(|r\pm s|)\succeq \min\{ \plog(r),\, \plog(s)\}$;
\item \label{item: plog(rs) = plog(r) + plog(s)} $\plog(rs)=\plog(r)+\plog(s)$; 
\item $\plog(r^k)=k\plog(r)$;
\item \label{item: order log(Q)} $\plog(r)\preceq \plog(s)$ if and only if $s/r\in \NN^+$. \qed 
 \end{enumerate}
\end{prop}

\begin{cor}
The multiplicative group of positive rationals is free-abelian of countably infinite rank; that is, $(\QQ^+,\cdot) \isom \big(\bigoplus_{p \in \PP}\ZZ, + \big)$. Moreover, the multiplicative group of rationals $(\QQ^*,\cdot)$ is isomorphic to $\ZZ/2\ZZ \times \bigoplus_{p\in \PP}\ZZ$. \qed
\end{cor}

We can use these extended $\PP$-logarithms to generalize the notions of $\gcd$ and $\lcm$ to the realm of rational numbers. To this end, we need to pass from 
$\ZZ$ to its order-completion obtained by adjoining the extreme element $-\infty$.



\begin{defn}
Let $R$ be a nonempty subset of $\QQ^+$. Then,
 \begin{enumerate}[(i)]
\item the \defin{logarithmic $\gcd$} of $R$ is $\lgcd(R) = \inf(\plog(R)) \in \Seq(\Zb)$;
\item the \defin{logarithmic $\lcm$} of $R$ is $\llcm(R) = \sup(\plog(R)) \in \Seq(\ZZ\cup\{+\infty\})$.
 \end{enumerate}
For notational convenience, we define 
 \[
\lgcd(R):=\lgcd\bigl(\{|r| \st r\in R\setmin\{0\}\}\bigr),
 \]
for every subset $\emptyset, \{0\}\neq R\subseteq \QQ$. Moreover, whenever these logarithmic extrema are of finite type (in particular, whenever~$R$ is finite), we define
 \begin{enumerate}[(i)]
\item the \defin{greatest common divisor} of $R$ as $\gcd(R)=\PP^{\,\lgcd(R)} \in \QQ^+$, and
\item the \defin{least common multiple} of $R$ as $\lcm(R)=\PP^{\,\llcm(R)} \in\QQ^+$.
\end{enumerate}
Of course, both definitions extend the corresponding notions for integers.
\end{defn}

\begin{rem}
Note that, by~\Cref{prop: Bezout N}, if $S\subseteq \ZZ^+$, then the generator of the cyclic subgroup $\gen{S}\leqslant \ZZ$ is $\gcd(S)=\PP^{\,\lgcd(S)}\in\ZZ$,
and hence $\gen{S}=\pm\PP^{\,\oSucc \lgcd(S)}\cup\{0\}$. When passing to the more general context $R\subseteq\QQ^+$, the expression $\gcd(R)= \PP^{\,\lgcd(R)}\in\QQ$ may no longer make sense, since one may have $\lgcd_p(R)=-\infty$ for some primes $p$, and/or $\lgcd_p(R)\neq 0$ for infinitely many $p$'s. Nevertheless, the equality $\gen{R}= \pm\PP^{\,\fSucc \lgcd(R)}\cup\{0\}$ still holds, as shown in~\Cref{prop: Bezout Q}. In this sense, \Cref{prop: Bezout Q} can be viewed as a rational analogue of Bezout's theorem: every subgroup $\gen{R}$ of $(\QQ,+)$, including non-(finitely generated) ones, is ``morally cyclic'' in the sense of being determined by a unique object, with $\lgcd(R)\in\Seq(\Zb)$ playing the role of the logarithmic generator; see also~\Cref{cor: lgcd R = lgcd H}.
\end{rem}

The logarithmic representation of rational numbers is illustrated in~\Cref{fig:compare-two-plogs}, where the logarithmic gcd and lcm appear as the coordinatewise minimum and maximum, respectively, of two $\PP$-logarithms.

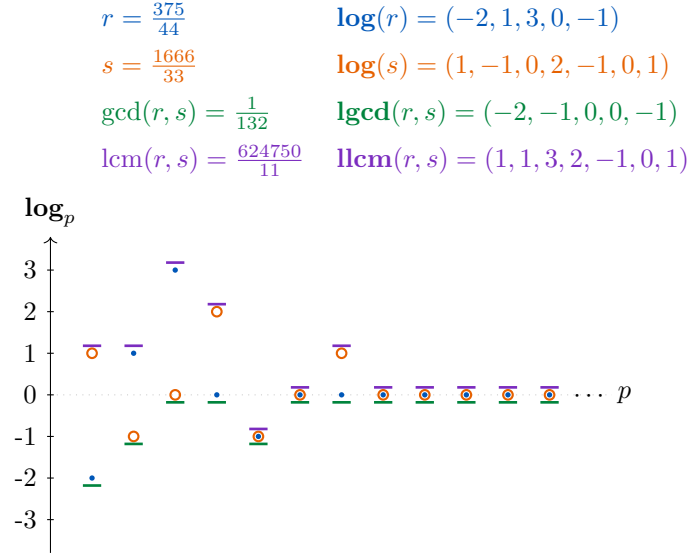
\begin{figure}[h]
\centering
\begin{tikzpicture}[
    x=0.55cm,y=0.55cm,
    axis/.style={->, thin},
    tick/.style={thin},
    every node/.style={font=\small}
]

\node[anchor=south west] at (0.6,5) {%
\begin{tabular}{ll}
\textcolor{colr}{$r=\tfrac{375}{44}$}
&
\textcolor{colr}{$\plog(r)=(-2,1,3,0,-1)$}
\\[2mm]
\textcolor{cols}{$s=\tfrac{1666}{33}$}
&
\textcolor{cols}{$\plog(s)=(1,-1,0,2,-1,0,1)$}
\\[2mm]
\textcolor{colgcd}{$\gcd(r,s)=\tfrac{1}{132}$}
&
\textcolor{colgcd}{$\lgcd(r,s)=(-2,-1,0,0,-1)$}
\\[2mm]
\textcolor{collcm}{$\operatorname{lcm}(r,s)=\tfrac{624750}{11}$}
&
\textcolor{collcm}{$\llcm(r,s)=(1,1,3,2,-1,0,1)$}
\end{tabular}
};

\draw[-, gray!40, thin, dotted] (0,0) -- (13.5,0);

\draw[axis] (0,-3.8) -- (0,3.8) node[above] {$\plog_p$};

\foreach \yy in {-3,-2,-1,0,1,2,3}
{
  \draw[tick] (-0.08,\yy) -- (0.08,\yy);
  \node[left] at (-0.08,\yy) {\yy};
}

\node[below] at (13,0.35) {$\cdots$};

\node at (13.8,0) {$p$};

\foreach \x/\h in {
    1/-2,2/1,3/3,4/0,5/-1,6/0,
    7/0,8/0,9/0,10/0,11/0,12/0
}
{
  \fill[colr] (\x,\h) circle[radius=1pt];
}

\foreach \x/\h in {
    1/1,2/-1,3/0,4/2,5/-1,6/0,
    7/1,8/0,9/0,10/0,11/0,12/0
}
{
  \draw[cols, line width=0.8pt]
    (\x,\h) circle[radius=1.8pt];
}

\foreach \x/\h in {
    1/-2,2/-1,3/0,4/0,5/-1,6/0,
    7/0,8/0,9/0,10/0,11/0,12/0
}
{
  \draw[colgcd, line width=1pt]
    (\x-0.22,\h-0.18) -- (\x+0.22,\h-0.18);
}

\foreach \x/\h in {
    1/1,2/1,3/3,4/2,5/-1,6/0,
    7/1,8/0,9/0,10/0,11/0,12/0
}
{
  \draw[collcm, line width=1pt]
    (\x-0.22,\h+0.18) -- (\x+0.22,\h+0.18);
}

\end{tikzpicture}
\caption{Comparison of two $\PP$-logarithms and their logarithmic $\gcd$ and $\lcm$}
\label{fig:compare-two-plogs}
\end{figure}

We finally obtain the desired extension of~\Cref{prop: Bezout N} to the rationals, which is a consequence of the following lemma.

\begin{lem} \label{lem: R iff R'}
Let $R$ be a nonempty subset of $\QQ^+$ and let $x \in \QQ^+$. Then, 
 $$
\plog(x) \succeq \lgcd(R) \ \Leftrightarrow\  \plog(x) \succeq \lgcd(R') \,\text{ for some finite } R'\subseteq R.
 $$
\end{lem}

\begin{proof}
The implication to the left is clear for every subset $R'\subseteq R$: indeed, if $R'\subseteq R$ then $\plog(R')\subseteq \plog(R)$, hence $\lgcd(R) =\inf(\plog(R))\preceq \inf(\plog(R'))=\lgcd(R')$.

For the implication to the right, assume that $\plog(x) \succeq \inf(\plog(R))$, \ie for every $p\in \PP$, $\plog_p(x) \geq \inf(\plog_p(R))$, and take $r_p \in R$ such that $\plog_p(x) \geq \plog_p(r_p)$. Now, take $r_0 \in R$, let $S=\supp(\plog(r_0)) \cup \supp(\plog(x))$, and define the finite subset
 $$
R'\,=\, \set{r_0}\cup \set{r_p \st p\in S} \subseteq R. 
 $$
Then, given a prime $p\in \PP$,
 \begin{itemize}
\item if $p\in S$ then $\plog_p(x) \geq \plog_p(r_p) \geq \inf(\plog_p(R'))$, 
\item if $p\notin S$ then $p \notin \supp(\plog(x))$ and $p \notin \supp(\plog(r_0))$; therefore, 
 $$ 
\plog_p(x) =0=\plog_p(r_0)\geq \inf(\plog_p(R')).
 $$
\end{itemize}
That is, $\plog(x) \succeq \inf (\plog(R'))$ with $R'$ finite, as we wanted to see.
\end{proof}

\begin{prop} \label{prop: Bezout Q}
Let $R$ be a nonempty subset of $\QQ^+$, and let $x\in \QQ^+$. Then,
 \begin{equation}\label{eq: Bezout Q}
x\in \gen{R} \ \Longleftrightarrow \ \plog(x) \succeq \lgcd(R) \,.
 \end{equation}
Equivalently, $\gen{R} =\pm \PP^{\,\fSucc \lgcd(R)} \cup \{0\}$.
\end{prop}



\begin{proof}
For any finite subset $R'=\set{r_1,\ldots,r_n} \subseteq \QQ^{+}$, we write $r_i=a_i/b_i$ in reduced form, and we denote by $B=\lcm(b_1,\dots,b_n)$, and by $A_i =Br_i =a_i \,\frac{B}{b_i}\in \NN^+$, $i=1,\ldots ,n$. Then:
\begin{align*}
    x \in \gen{R}
    &\,\Leftrightarrow\,
    x \in \gen{R'} 
    \text{, for some finite } R'=\{r_1,\ldots, r_n \}\subseteq R\\
    &\,\Leftrightarrow\,
    B x \in \gen{Br_1,\ldots ,Br_n} \text{ (where $ Br_i\in \NN^+$)} 
    \\
    &\,\Leftrightarrow\,
    \plog(B x) \succeq \inf(\set{\plog(Br_i) \st i\in[n]}) \\
    &\,\Leftrightarrow\,
    \plog(B) + \plog(x) \succeq \inf(\set{\plog(B) + \plog(r_i) \st i \in [n] }) \\
    &\,\Leftrightarrow\,
    \plog(B) + \plog(x) \succeq \plog(B) + \inf(\plog(R')) \\
    &\,\Leftrightarrow\,
    \plog(x) \succeq \lgcd(R') 
    \,\text{, where $R' \subseteq R$ is finite}\\
    &\,\Leftrightarrow\,
    \plog(x) \succeq \lgcd(R) \, ,
\end{align*}
where we have used~\eqref{eq: Bezout} in the third equivalence, and~\Cref{lem: R iff R'} in the last one. This proves~\eqref{eq: Bezout Q}. Finally, the equality $\gen{R} =\pm \PP^{\,\fSucc \lgcd(R)} \cup \{0\}$ is just the reformulation of~\eqref{eq: Bezout Q} via the bijection $\plog\colon \QQ^+\to \Seq_0(\ZZ)$. 
\end{proof}


\begin{cor}\label{cor: lgcd R = lgcd H}
Let $R,S\subseteq\QQ$ be nonempty and different from $\{0\}$. Then, \[
\lgcd(R)=\lgcd(S)
\ \Leftrightarrow \ 
\gen{R}=\gen{S} \,.  \tag*{\qedsymbol}
\]
\end{cor}



\section{The lattice of subgroups of $(\QQ, +)$}

We already have the ingredients to extend the logarithmic correspondence given by $\plog$ to the level of subgroups of~$\QQ$, and obtain an explicit classification of this family. A similar goal was already achieved in~\cite{beaumontCharacterizationSubgroupsAdditive1951a}. It will also transfer the order structure and the operations of infimum, supremum, and sum on sequences to natural operations on subgroups. Moreover, it will help to deal with further questions concerning isomorphism classes and intersections.

Let us denote by $\Sgp(\QQ,+)$ the family of subgroups of the additive group of rationals. Write $\Sgp^*(\QQ,+)$ for the set of nontrivial ones.

\begin{lem}\label{lem: lgcd eventually negative}
For every $r\in R\subseteq \QQ^+$, we have $\supp_{>0}(\lgcd(R)) \subseteq \supp(\plog(r))$. In particular, $\supp_{>0}(\lgcd(R))$ is always finite; equivalently, $\lgcd(R) \in \Seq_{\leq0}(\Zb)$.
\end{lem}

\begin{proof}
Note that, for every $p\in \PP$,
 \begin{align*}
p\in \supp_{>0}(\lgcd(R))
&\,\Leftrightarrow\,
\lgcd_p(R) > 0\\
&\,\Leftrightarrow\,
\inf(\plog_p(R)) >0 \\
&\,\Rightarrow\,
\plog_p(r) > 0 \\
&\,\Rightarrow\,
p\in \supp(\plog(r)) \, .
 \end{align*}
Since $\supp(\plog(r))$ is finite for every $r \in \QQ^+$, the  final claim follows immediately.
\end{proof}

\begin{thm}\label{thm: lgcd structure}
Let $(\QQ,+)$ be the additive group of rational numbers, and let $\Seq_{\leq0}(\Zb)$ denote the set of eventually nonpositive sequences in $\Zb = \ZZ \cup \set{-\infty}$. Then, the nontrivial subgroups of $(\QQ,+)$ are in bijection with $\Seq_{\leq0}(\Zb)$ via
 \begin{equation}\label{eq: bij lgcd}
\begin{array}{rcl} \Sgp^*(\QQ,+) & \rightarrow & \Seq_{\leq0}(\Zb) \\ H & \mapsto &   \lgcd(H) \\ H_\s{s} & \mapsfrom & \s{s} \,, \end{array}
 \end{equation}
where $H_\s{s}=\pm\PP^{\,\fSucc \s{s}} \cup \set{0}=\pm \set{\PP^{\,\s{e}}:\s{e}\in\Seq_0(\ZZ) \text{ and } \s{s}\preceq \s{e}} \cup \set{0}$.
Moreover, the map~\eqref{eq: bij lgcd} is an anti-isomorphism of lattices; that is, for every $H,K\in \Sgp^*(\QQ,+)$,
 \begin{equation}\label{eq: subgroup inclusion}
H\leqslant K \;\Leftrightarrow\; \lgcd(H) \succeq \lgcd(K).    
 \end{equation}
\end{thm}

\begin{proof}
By~\Cref{lem: lgcd eventually negative}, the direct map $H\mapsto \lgcd(H)$ in \eqref{eq: bij lgcd} is well defined. For the converse, let us see that $H_{\s{s}}$ is a nontrivial subgroup of $(\QQ,+)$. 

By construction, $0\in H_{\s{s}}$ and $H_{\s{s}}$ is closed under change of sign. Let us see that if $x,y\in H_{\s{s}}$ then $x+y\in H_{\s{s}}$. If~$0\in \set{x,y}$ there is nothing to prove. Otherwise, write $x=\varepsilon x_0$ and $y=\delta y_0$ with $\varepsilon, \delta\in\{\pm 1\}$ and $x_0,y_0\in H_{\s{s}}\cap\QQ^+$. If $\varepsilon =\delta$, then $x+y=\varepsilon (x_0+y_0)$. If $\varepsilon \neq \delta$, then $x+y=\varepsilon (x_0-y_0)$. So, it is enough to assume $x,y\in H_{\s{s}}\cap\QQ^+$ and prove that $|x\pm y| \in H_{\s{s}}$.

Indeed, under these assumptions, if $|x\pm y|=0$ the claim is obvious; otherwise, we have that $\plog(x),\plog(y)\in\Seq_0(\ZZ)$ with both $\plog(x), \plog(y) \succeq \s{s}$. But in this case we have that $x\pm y\in \gen{x,y}$ and, by~\Cref{prop: plog(Q)}.\ref{item: logsuma},
 $$
\plog(|x \pm y|) \succeq 
\min(\plog(x),\plog(y)) \succeq \s{s} \, .
 $$
This shows that $H_{\s{s}}$ is a subgroup of $(\QQ,+)$. Moreover, $H_{\s{s}}\neq \{0\}$ because it always contains a big enough integer: since $\s{s}\in \Seq_{\leq0}(\Zb)$, $\supp_{>0} (\s{s})$ is finite and $0\neq \prod_{p\in \supp_{>0}(\s{s})} p^{\s{s}(p)}\in \ZZ\cap H_{\s{s}}$.

Let us now see that the two maps from~\eqref{eq: bij lgcd} are inverse of each other. Indeed, for every $\s{s} \in \Seq_{\leq0}(\Zb)$,
 $$
\lgcd(H_{\s{s}}) \,=\lgcd(H_{\s{s}}\cap\QQ^+) \,=\, \lgcd(\PP^{\,\fSucc  \s{s}}) \,=\, \inf(\plog(\PP^{\,\fSucc  \s{s}})) \,=\, \inf(\fSucc  \s{s}) \,=\, \s{s}.
 $$
On the other hand, for $H\in \Sgp^*(\QQ,+)$, taking $R=H\cap \QQ^+$ in~\Cref{prop: Bezout Q}, we have that $\PP^{\fSucc \lgcd(H)}=H\cap \QQ^+ \,,$ and hence
 $$
H_{\lgcd(H)} \,=\, \pm\PP^{\,\fSucc  \lgcd(H)} \cup \set{0} \,=\, \pm (H \cap \QQ^+) \cup \set{0} \,=\, H.
 $$

Finally, given two nontrivial subgroups of $(\QQ, +)$, say $H=H_{\lgcd(H)}$ and $K=H_{\lgcd(K)}$, we can use again~\Cref{prop: Bezout Q} and get   
 $$
H\leqslant K \;\Leftrightarrow\; H_{\lgcd(H)}\leqslant H_{\lgcd(K)} \;\Leftrightarrow\; \lgcd(K)\preceq \lgcd(H).
 $$
This completes the proof.
\end{proof}

Note that the bijection $\lgcd\colon \Sgp^*(\QQ,+) \rightarrow \Seq_{\leq0}(\Zb)$ may be viewed as a natural extension of the bijection $\plog \colon \QQ^+ \to \Seq_0(\ZZ)$, at the price of identifying the finitely generated (\ie cyclic) subgroups $H$ with their positive generators $\PP^{\lgcd(H)} \in \QQ^+$. 
Furthermore, this correspondence provides a transparent way to transport the notions of support and distance from logarithmic sequences to subgroups of $(\QQ,+)$, a terminology that will be convenient throughout the rest of the paper.

\begin{defn}
For every nontrivial subgroup $H\leqslant (\QQ,+)$, every rational $r\in \QQ^+$, and every subset $B \subseteq \Zb$, we define the \defin{$B$-support} of $H$ and $r$ via their associated logarithmic sequences. More precisely,
 \begin{align*}
\supp_B(H) &\,=\, \supp_B(\lgcd(H)) \\ \supp_B(r) &\,=\, \supp_B(\plog(r)) \,.
 \end{align*}  
Note that if $n\in \NN^+$, then $\supp(n)
$ is exactly the set of primes in the prime factorization of $n$.

Similarly, the distance between two nontrivial subgroups $H,K \leqslant (\QQ,+)$ is defined to be the distance between their respective $\lgcd$'s:
\[
\dist(H,K) \,=\, \dist(\lgcd(H),\lgcd(K)) \,.
\]
\end{defn}

Across bijection~\eqref{eq: bij lgcd} one can naturally understand the index of an extension of subgroups $H\leqslant K\leqslant \QQ$, and the elementary operations among subgroups of $\QQ$. 

\begin{prop}
Let $\set{0} \neq H\leqslant K\leqslant \QQ$. Then, the index of $H$ in $K$ is\footnote{With the natural conventions that $(-\infty)-(-\infty)=0$, and that the product of infinite primes, including $p^{\infty}$, equals $\infty$.} 
 \begin{equation}\label{eq: index}
|K:H|\, =\, \PP^{\,\lgcd(H)-\lgcd(K)}.
 \end{equation}
In particular, 
\[
|K:H|<\infty
\ \Leftrightarrow\ 
\lgcd(H)-\lgcd(K)\in \Seq_0(\NN)
\ \Leftrightarrow\ 
\dist(H,K) < \infty \,.
\]
\end{prop}

\begin{proof}
To prove the index formula \eqref{eq: index} we first claim that, for every prime $p\in \PP$, and every nontrivial subgroup $H\leqslant \QQ$, we have $pH\leqslant H$ with  
 \begin{equation} \label{eq: index H:pH}
|H:pH|= \bigg\{\! \begin{array}{ll} p & \text{if } \lgcd_p(H)>-\infty, \\ 1 & \text{if } \lgcd_p(H)=-\infty. \end{array}
 \end{equation}
Indeed, if $\lgcd_p(H)=-\infty$, then $pH=H$, and the claim is trivial. Otherwise, for any fixed $h\in H$ such that $\plog_p(h)=\lgcd_p(H) \in \ZZ$, the map
 $$
\begin{array}{rcl} \theta\colon  H & \to & \ZZ/p\ZZ \\ x &\mapsto &x/h \,(\mathrm{mod}\ p) \end{array}
 $$
is a well-defined epimorphism with kernel $pH$. Indeed, for every $x\in H \leqslant \QQ$, $\plog_{p}(x/h)=\plog_{p}(x)-\plog_{p}(h) \geq 0$. Therefore, $x/h =m/n$, with $m,n\in \NN$ and $n$ invertible modulo $p$; hence, $x/h \,(\mathrm{mod}\ p)=mn^{-1} \,(\mathrm{mod}\ p)$ is well defined. Finally, since $h\in H$ and $\theta(h)=1$, $\theta$ is an epimorphism and
 \begin{align*}
x\in\ker(\theta)
&\,\Leftrightarrow\,
x/h \equiv 0 \,(\mathrm{mod}\ p)\\
&\,\Leftrightarrow\,
\plog_p(x/h)\geq 1\\
&\,\Leftrightarrow\,
\plog_p(x)\geq \plog_p(h) + 1\\
&\,\Leftrightarrow\,
x \in pH.   
 \end{align*}
We conclude that $\ker(\theta)=pH$ and so, $|H:pH|=p$, as claimed in~\eqref{eq: index H:pH}.

Now let $H\leqslant K$, and write $\s{h}=\lgcd(H)$ and $\s{k}=\lgcd(K)$, both sequences from $\Seq_{\leq0}(\Zb)$. By~\eqref{eq: subgroup inclusion}, we have that $\s{k}\preceq \s{h}$; hence, $\s{d}=\s{h}-\s{k}$ (with the conventions done) is a sequence with entries in $\NN\cup\{\infty\}$. We distinguish two cases, depending on whether $\s{d} \in \Seq_0(\NN)$ or not. 

Assume $\s{d} \in \Seq_0(\NN)$, \ie $\s{d}=\sum_{p\in P} d_p \mathbf{1}_p$; see~\eqref{eq: indicator}. Then, $H=H_\s{h}=H_{\s{k}+\s{d}}= H_{\s{k}+\mathord{\scalebox{0.7} {$\sum$}}_{p\in P} d_p \mathbf{1}_p}$. Now, starting from $\s{k}$ and increasing by 1 each of its coordinates, one at a time, until reaching $\s{h}$, we obtain a list of sequences $\s{k}=\s{s_0}\preceq \s{s_1}\preceq \cdots \preceq \s{s_n}=\s{h}$ (from $\Seq_{\leq0}(\Zb)$, and with $n=\sum_{p\in \PP} d_p$), corresponding to the strict inclusions of subgroups $H=H_{\s{h}}=H_{\s{s_n}}<H_{\s{s_{n-1}}}<\cdots <H_{\s{s_1}}<H_{\s{s_0}}=H_{\s{k}}=K$, each of the form $pL<L$ for some prime $p$ and with strict inclusion. By~\eqref{eq: index H:pH}, and the multiplicativity of indices,
 \[
|K:H|=\prod_{i=n-1}^0 |H_{\s{s_i}}:H_{\s{s_{i+1}}}|=\prod_{p\in P} p^{d_p}=\PP^{\,\s{h}-\s{k}}=\PP^{\,\lgcd(H)-\lgcd(K)}.
 \]

Assume now $\s{d} = \s{h} - \s{k} \notin \Seq_0(\NN)$. This means that, either infinitely many coordinates satisfy $h_p>k_p$, or there exists $p\in\PP$ such that $k_p=-\infty$ and $h_p\in\ZZ$. In either case, for every $n\in\NN$ we can choose a finite sequence (of sequences) $(\s{s_i})_{i=0}^n$ such that $\s{k} \preceq \s{s_0}\preceq \s{s_1}\preceq \cdots \preceq \s{s_n}\preceq \s{h}$, where each $\s{s_{j+1}}$ is obtained from $\s{s_j}$ by increasing a single coordinate by~$1$. Therefore, by~\eqref{eq: index H:pH}, $|H_{\s{s}_j}:H_{\s{s}_{j+1}}|\geq 2$ for every $j=0,\ldots ,n-1$ and so, $|K:H|\geq 2^n$. Since this is true for every $n\in \NN$, we conclude that $|K:H|=\infty =\PP^{\,\lgcd(H)-\lgcd(K)}$ (under the conventions done). This proves~\eqref{eq: index}. 

By the stated convention for the product $\PP^{\,\lgcd(H)-\lgcd(K)}$, it follows that $|K:H|<\infty \ \Leftrightarrow\ \lgcd(H)-\lgcd(K)\in\Seq_0(\NN)$. Moreover, since $H\leqslant K$ implies
$\lgcd(K)\preceq\lgcd(H)$, the definition of distance between
subgroups yields the final assertion.
\end{proof}

\begin{prop}\label{prop: =}
For every $r\in\QQ^+$ and every~$H,K \in \Sgp^*(\QQ,+)$, we have
 \begin{align}
\lgcd(rH) &=
\plog(r)+\lgcd(H), \label{eq: lgcd scalar} \\
\lgcd(H+K) &= \min\bigl(\lgcd(H),\,\lgcd(K)\bigr), \label{eq: lgcd sum} \\
\lgcd(H\cap K) &= \max\bigl(\lgcd(H),\,\lgcd(K)\bigr), \label{eq: lgcd cap} \\ 
\lgcd(HK) &= \lgcd(H)+\lgcd(K), \label{eq: lgcd prod}
 \end{align}
where $HK=\{xy \mid x\in H,\ y\in K\}$ is, again, a subgroup of $(\QQ, +)$.
\end{prop}

\begin{proof} 
For all the proof, fix two nontrivial subgroups $\{0\}\neq H,K\leqslant \QQ$, and let $\s{h}=\lgcd(H)$ and $\s{k}=\lgcd{K}$; that is, $H=H_{\s{h}}$ and $K=H_{\s{k}}$.

For every $x\in H\setmin\{0\}$, $\plog(rx)=\plog(r)+\plog(x)$. Taking infima, equality~\eqref{eq: lgcd scalar} follows immediately.

Using~\eqref{eq: subgroup inclusion} and since $H,K\leqslant H+K$, we get $\lgcd(H+K)\preceq \s{h}$ and $\lgcd(H+K)\preceq \s{k}$; hence, $\lgcd(H+K)\preceq \min(\s{h},\s{k})$. Conversely, $\min(\s{h},\s{k})\preceq \s{h},\s{k}$ and so, $H=H_{\s{h}}\leqslant H_{\min(\s{h},\s{k})}$ and $K=H_{\s{k}}\leqslant H_{\min(\s{h},\s{k})}$. Therefore $H+K\leqslant H_{\min(\s{h},\s{k})}$ and so, again by~\eqref{eq: subgroup inclusion}, $\min(\s{h},\s{k})=\lgcd(H_{\min(\s{h},\s{k})})\preceq \lgcd(H+K)$. Thus, $\lgcd(H+K)=\min(\s{h},\s{k})=\min\bigl(\lgcd(H),\lgcd(K)\bigr)$, as claimed in \eqref{eq: lgcd sum}.

A completely analogous argument shows~\eqref{eq: lgcd cap}. 

Finally, to see~\eqref{eq: lgcd prod} observe that 
 $$
HK=H_{\s{h}}H_{\s{k}}=\{hk \st h\in H_{\s{h}},\, k\in H_{\s{k}}\}=\pm \PP^{\fSucc \s{h}} \cdot \PP^{\fSucc \s{k}}\cup \{0\}= 
 $$
 $$
=\pm \PP^{\fSucc \s{h}+\fSucc \s{k}}\cup \{0\}=\pm \PP^{\fSucc (\s{h}+\s{k})}\cup \{0\}=H_{\s{h}+\s{k}}.
 $$
This automatically shows~\eqref{eq: lgcd prod} and that $HK$ is, certainly, a subgroup. 
\end{proof}

\begin{rem}
The fact that $HK=\{hk \st h\in H,\, k\in K\}$ is again an additive subgroup of $\QQ$ is specific to this setting; it is not a general group-theoretic phenomenon. An alternative argument to see this, in classical terms, is as follows: take $hk,\, h'k'\in HK$ and consider $hk+h'k'$. By the cyclicity of all finitely generated subgroups of $\QQ$, there exist $h''\in H$ and $k''\in K$ such that $\gen{h,\, h'}=\gen{h''}$ and $\gen{k,\, k'}=\gen{k''}$; in particular, $h=\lambda h''$, $h'=\lambda' h''$, $k=\mu k''$, and $k'=\mu' k''$ for some integers $\lambda,\, \lambda', \mu, \mu'\in \ZZ$. Then, 
 $$
hk+h'k'=\lambda h'' \mu k'' +\lambda' h'' \mu' k''=h''k''(\lambda \mu +\lambda'\mu') =h''\Big((\lambda \mu +\lambda'\mu')k''\Big)\in HK.
 $$
\end{rem}

The identities obtained in~\Cref{prop: =} extend naturally to arbitrary families of subgroups with straightforward inductive arguments.

\begin{cor}
Let $\s{s_1},\ldots,\s{s_k} \in \Seq_{\leq0}(\Zb)$. Then,
 \begin{enumerate} [(i)]
\item $H_{\s{s_1}}+\cdots +H_{\s{s_k}} =H_{\min (\s{s_1},\ldots,\s{s_k})}$, \label{item: sum Hs}
\item $H_{\s{s_1}}\cap \cdots \cap H_{\s{s_k}} =H_{\max (\s{s_1},\ldots,\s{s_k})}$, \label{item: cap Hs}
\item $H_{\s{s_1}}\cdots H_{\s{s_k}}=H_{\s{s_1}+\cdots +\s{s_k}}$. 
\qed \label{item: prod Hs}
 \end{enumerate}
\end{cor}

\begin{cor}
Let $\mathcal{S}$ be a nonempty set of sequences in $\Seq_{\leq0}(\Zb)$. Then,
 \begin{enumerate}[(i)]
\item $\sum_{\s{s}\in \mathcal{S}} H_{\s{s}} = H_{\inf (\mathcal{S})}$.
\item $\bigcap_{\s{s}\in \mathcal{S}} H_{\s{s}}=\begin{cases} H_{\sup(\mathcal S)} & \text{if } \sup(\mathcal S)\in \Seq_{\leq0}(\Zb), \\ \{0\} & \text{otherwise}. \end{cases}$ \qed
 \end{enumerate}
\end{cor}

We conclude this discussion with several examples illustrating the different shapes that logarithmic gcd profiles of subgroups of $(\QQ,+)$ may take; see~\Cref{fig:lgcd-subgroups-types}.

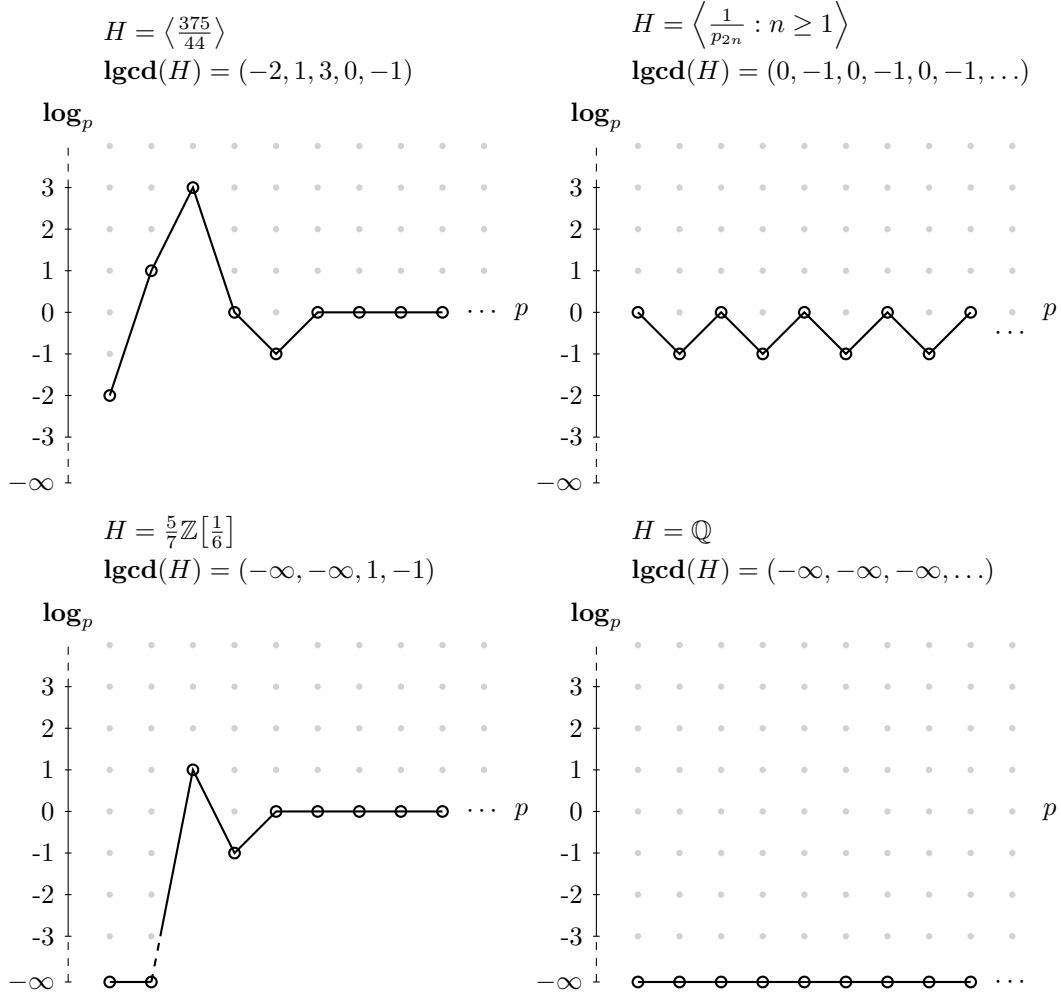
\begin{figure}[h]
\centering
\begin{tikzpicture}[
    x=0.55cm,y=0.55cm,
    axis/.style={thin},
    tick/.style={thin},
    prof/.style={black, line width=0.8pt},
    upperdot/.style={gray!35, fill=gray!35},
    every node/.style={font=\small}
]

\def\ytop{4.1}

\newcommand{\panelaxes}{
    \draw[axis] (0,-3) -- (0,3);
    \draw[dashed, thin] (0,3) -- (0,\ytop);
    \draw[dashed, thin] (0,-4.1) -- (0,-3);
    \node[above] at (0,\ytop) {$\plog_p$};

    \foreach \yy in {-3,-2,-1,0,1,2,3}
    {
      \draw[tick] (-0.08,\yy) -- (0.08,\yy);
      \node[left] at (-0.08,\yy) {\yy};
    }

    \draw[tick] (-0.08,-4.1) -- (0.08,-4.1);
    \node[left] at (-0.08,-4.1) {$-\infty$};

    \node at (10.9,0) {$p$};
}

\begin{scope}[shift={(0,12)}]
\panelaxes

\node[anchor=south west, align=left] at (0.6,5.2) {%
    \textcolor{black}{$H=\left\langle \tfrac{375}{44}\right\rangle$}\\[1mm]
    \textcolor{black}{$\lgcd(H)=(-2,1,3,0,-1)$}
};

\foreach \y in {-1,0,1,2,3,4}
{
  \fill[upperdot] (1,\y) circle[radius=1.2pt];
}
\foreach \y in {2,3,4}
{
  \fill[upperdot] (2,\y) circle[radius=1.2pt];
}
\foreach \y in {4}
{
  \fill[upperdot] (3,\y) circle[radius=1.2pt];
}
\foreach \y in {1,2,3,4}
{
  \fill[upperdot] (4,\y) circle[radius=1.2pt];
}
\foreach \y in {0,1,2,3,4}
{
  \fill[upperdot] (5,\y) circle[radius=1.2pt];
}
\foreach \x in {6,7,8,9}
{
  \foreach \y in {1,2,3,4}
  {
    \fill[upperdot] (\x,\y) circle[radius=1.2pt];
  }
}
\foreach \y in {1,2,3,4}
{
  \fill[upperdot] (10,\y) circle[radius=1.2pt];
}

\draw[prof]
  (1,-2) -- (2,1) -- (3,3) -- (4,0) -- (5,-1) --
  (6,0) -- (7,0) -- (8,0) -- (9,0);

\foreach \x/\h in {1/-2,2/1,3/3,4/0,5/-1,6/0,7/0,8/0,9/0}
{
  \draw[prof] (\x,\h) circle[radius=2pt];
}

\node[prof] at (10,0) {$\cdots$};
\end{scope}

\begin{scope}[shift={(12.7,12)}]
\panelaxes

\node[anchor=south west, align=left] at (0.6,5.2) {%
    \textcolor{black}{$H=\left\langle \tfrac1{p_{2n}} : n\ge 1\right\rangle$}\\[1mm]
    \textcolor{black}{$\lgcd(H)=(0,-1,0,-1,0,-1,\ldots)$}
};

\foreach \x in {1,3,5,7,9}
{
  \foreach \y in {1,2,3,4}
  {
    \fill[upperdot] (\x,\y) circle[radius=1.2pt];
  }
}
\foreach \x in {2,4,6,8}
{
  \foreach \y in {0,1,2,3,4}
  {
    \fill[upperdot] (\x,\y) circle[radius=1.2pt];
  }
}
\foreach \y in {0,1,2,3,4}
{
  \fill[upperdot] (10,\y) circle[radius=1.2pt];
}

\draw[prof]
  (1,0) -- (2,-1) -- (3,0) -- (4,-1) -- (5,0) --
  (6,-1) -- (7,0) -- (8,-1) -- (9,0);

\foreach \x/\h in {1/0,2/-1,3/0,4/-1,5/0,6/-1,7/0,8/-1,9/0}
{
  \draw[prof] (\x,\h) circle[radius=2pt];
}

\node[prof] at (10,-0.5) {$\cdots$};
\end{scope}

\begin{scope}[shift={(0,0)}]
\panelaxes

\node[anchor=south west, align=left] at (0.6,5.2) {%
    \textcolor{black}{$H=\tfrac{5}{7}\ZZ\!\left[\tfrac16\right]$}\\[1mm]
    \textcolor{black}{$\lgcd(H)=(-\infty,-\infty,1,-1)$}
};

\foreach \x in {1,2}
{
  \foreach \y in {-3,-2,-1,0,1,2,3,4}
  {
    \fill[upperdot] (\x,\y) circle[radius=1.2pt];
  }
}
\foreach \y in {2,3,4}
{
  \fill[upperdot] (3,\y) circle[radius=1.2pt];
}
\foreach \y in {0,1,2,3,4}
{
  \fill[upperdot] (4,\y) circle[radius=1.2pt];
}
\foreach \x in {5,6,7,8,9}
{
  \foreach \y in {1,2,3,4}
  {
    \fill[upperdot] (\x,\y) circle[radius=1.2pt];
  }
}
\foreach \y in {1,2,3,4}
{
  \fill[upperdot] (10,\y) circle[radius=1.2pt];
}

\draw[prof] (1,-4.1) -- (2,-4.1);
\draw[prof,dashed] (2,-4.1) -- (2.216,-3);
\draw[prof]
  (2.216,-3) -- (3,1) -- (4,-1) -- (5,0) --
  (6,0) -- (7,0) -- (8,0) -- (9,0);

\foreach \x in {1,2}
{
  \draw[prof] (\x,-4.1) circle[radius=2pt];
}

\foreach \x/\h in {3/1,4/-1,5/0,6/0,7/0,8/0,9/0}
{
  \draw[prof] (\x,\h) circle[radius=2pt];
}

\node[prof] at (10,0) {$\cdots$};
\end{scope}

\begin{scope}[shift={(12.7,0)}]
\panelaxes

\node[anchor=south west, align=left] at (0.6,5.2) {%
    \textcolor{black}{$H=\QQ$}\\[1mm]
    \textcolor{black}{$\lgcd(H)=(-\infty,-\infty,-\infty,\ldots)$}
};

\foreach \x in {1,2,3,4,5,6,7,8,9}
{
  \foreach \y in {-3,-2,-1,0,1,2,3,4}
  {
    \fill[upperdot] (\x,\y) circle[radius=1.2pt];
  }
}
\foreach \y in {-3,-2,-1,0,1,2,3,4}
{
  \fill[upperdot] (10,\y) circle[radius=1.2pt];
}

\draw[prof]
  (1,-4.1) -- (2,-4.1) -- (3,-4.1) -- (4,-4.1) --
  (5,-4.1) -- (6,-4.1) -- (7,-4.1) -- (8,-4.1) -- (9,-4.1);

\foreach \x in {1,2,3,4,5,6,7,8,9}
{
  \draw[prof] (\x,-4.1) circle[radius=2pt];
}

\node[prof] at (10,-4.1) {$\cdots$};
\end{scope}

\end{tikzpicture}
\caption{Logarithmic gcd profiles for different kinds of subgroups of $(\QQ,+)$}
\label{fig:lgcd-subgroups-types}
\end{figure}

Also, as a corollary we obtain the well-known description of the finitely generated subgroups of $(\QQ,+)$. 
Following the usual convention in geometric group theory, we define the \defin{rank} of a group $G$, denoted by $\rk(G)$, as the minimum cardinality of a generating set for $G$. 

\begin{cor} \label{cor: fg iff}
Let $H$ be a nontrivial subgroup of $(\QQ,+)$. Then, the following are equivalent:
 \begin{enumerate}[(a)]
\item\label{item: cyclic} $H$ is cyclic,
\item\label{item: fg} $H$ is finitely generated,
\item\label{item: ftype} $\lgcd(H)$ is of finite type.
 \end{enumerate}
\end{cor}

\begin{proof}
The implication $[\ref{item: cyclic} \Rightarrow \ref{item: fg}]$ is trivial.

For $[\ref{item: fg} \Rightarrow \ref{item: ftype}]$, assume that $H$ is finitely generated, say $H=\gen{R}$ for some finite set $R\subseteq \QQ$. Replacing each generator by its absolute value if necessary, we may assume that $R\subseteq \QQ^+$. Since $R$ is finite, $\lgcd(H)=\lgcd(R)=\min(\plog(R))\in \Seq_0(\ZZ)$ and this is of finite type.


Finally, to see $[\ref{item: ftype} \Rightarrow \ref{item: cyclic}]$, note that if $\lgcd(H) \in \Seq_0(\ZZ)$ then, by \Cref{prop: Bezout Q}, $H=\gen{H^+}=\pm \PP^{\,\fSucc \lgcd(H^+)}\cup \{0\}=\gen{h}$, where $h=\PP^{\, \lgcd(H^+)}\in \QQ$ is its positive generator. 
\end{proof}

\begin{rem} \label{rem: subgroup ranks}
Only three different ranks --- namely $0$, $1$, and $\infty$ --- can occur for a subgroup $H$ of $(\QQ,+)$:
 \begin{enumerate}[(a)]
\item\label{item: rk 0} if $H=\set{0}$ then $\rk(H)=0$;
\item\label{item: rk 1} if $\lgcd(H)$ is of finite type then $H=\gen{\gcd(H)}$, and therefore $\rk(H)=1$;
\item\label{item: rk infty} if $\lgcd(H)$ is of infinite type then $\rk(H)=\infty$. This will be the case in the following two (non-exclusive) situations:
 \begin{enumerate}[label=(c.\roman*)]
\item infinitely many entries in $\lgcd(H)$ are negative, \ie $|\supp(H)|=\infty$; in this case, we say that $H$ is a non-(finitely generated) subgroup of \defin{horizontal type}. 
\item at least one entry in $\lgcd(H)$ is equal to $-\infty$, \ie $-\infty \in \lgcd(H)$; in this case, we say that $H$ is a non-(finitely generated) subgroup of \defin{vertical type}. 
 \end{enumerate}
 \end{enumerate}
For example, $\gen{1/p \st p\in \PP}=\gen{1/2, 1/3, 1/5, \ldots}\leqslant \QQ$ is a non-(finitely generated) subgroup of horizontal and non-vertical type, $\gen{1/2^n \st n\geq 1}=\gen{1/2, 1/4, 1/8, \ldots }\leqslant \QQ$ is a non-(finitely generated) subgroup of vertical and non-horizontal type, and $\QQ=\gen{1/p^n \st p\in \PP,\, n\geq 1}$ is a non-(finitely generated) subgroup of both horizontal and vertical type.
\end{rem}


\begin{exm}[Localizations of $\ZZ$]
Let $P\subseteq \PP$ be any set of primes, and consider the subgroup
 $$
\ZZ[P^{-1}]:=\Set{\frac{a}{n}\in \QQ \st a\in\ZZ,\ n\in\NN^+,\ \supp(n)\subseteq P }\leqslant \QQ.
 $$
Then, $\ZZ[P^{-1}]$ is a subgroup of $(\QQ,+)$, called the \defin{localization of $\ZZ$ at 
$P$}; we have $\lgcd\bigl(\ZZ[P^{-1}]\bigr)=-\infty\cdot \mathbf{1}_{P}$. 
Here are some relevant particular cases:
 \begin{itemize}
\item if $P=\varnothing$ then $\ZZ[P^{-1}]=\ZZ$ is finitely generated (in fact, cyclic), with $\lgcd(\ZZ)=\s{0}$;
\item if $n=p_1^{e_1}\cdots p_k^{e_k}\in \NN_{\geq 2}$ and $P=\supp(n)=\{p_1,\ldots ,p_k\}$, then 
 $$
\ZZ[P^{-1}]=\ZZ\!\left[\tfrac{1}{n}\right]=\Set{\frac{a}{n^m}\st a\in\ZZ,\ m\in\NN}=\ZZ\!\left[\tfrac1{p_1},\ldots,\tfrac1{p_k}\right]
 $$
is non-(finitely generated) of vertical type, with $\lgcd(\ZZ[1/n])=-\infty\cdot \mathbf{1}_{P}$. 
\item if $|P|=\infty$ then $\ZZ[P^{-1}]=\ZZ[1/p, p\in P]$ is non-(finitely generated) of both horizontal and vertical types, with $\lgcd(\ZZ[P^{-1}])=-\infty\cdot \mathbf{1}_{P}$. In particular, $\ZZ[\PP^{-1}]=\QQ$ with $\lgcd(\QQ)=-\infty\cdot \mathbf{1}_{\PP}$.
\end{itemize}
\end{exm}

We now apply the logarithmic description of subgroups of $(\QQ,+)$ to study the multiplicative structures that they may inherit from~$\QQ$. More precisely, we characterize when a subgroup of~$(\QQ,+)$ is closed under multiplication, when does it contain the multiplicative identity, and when does it define a subring or an ideal of a localization $\ZZ[P^{-1}]$. The logarithmic viewpoint turns these algebraic conditions into simple coordinatewise restrictions on the associated
sequence~$\lgcd(H)$.

\begin{lem}
Let $H,K$ be nontrivial subgroups of $(\QQ,+)$. Then,
 \begin{enumerate}[(i)]
\item\label{item: 1 in H} $1\in H$ if and only if $\lgcd(H) \preceq \s{0}$;
\item\label{item: HH in H} $HH\subseteq H$ if and only if $\lgcd_p(H) \in \{-\infty\}\cup \NN$ for every $p\in \PP$;
\item\label{item: KH in H} $KH \subseteq H$ if and only if $\lgcd_p(K) \in \NN$ for every $p\in \supp_{\ZZ}(H)$.
 \end{enumerate}
\end{lem}

\begin{proof}
Let $\s{h}=\lgcd(H)$ and $\s{k}=\lgcd(K)$. 

Since $\plog(1)=\s{0}$, (i) follows immediately from \Cref{prop: Bezout Q}.

By~\eqref{eq: lgcd prod}, $HH\subseteq H$ if and only if $\s{h}\preceq 2\s{h}$, \ie if and only if $h_p\leq 2h_p$ for every $p\in\PP$. This is equivalent to
$h_p\in\{-\infty\}\cup\NN$ for every $p\in\PP$, proving~(ii).
    
Again by~\eqref{eq: lgcd prod}, $KH\subseteq H$ if and only if $\s{h}\preceq \s{k}+\s{h}$, \ie if and only if, for every $p\in \PP$, either $h_p=-\infty$ or $k_p\in\NN$. This proves~(iii).
\end{proof}

\begin{cor}
Let $H$ be a nontrivial subgroup of $(\QQ,+)$. Then,
 \begin{enumerate}[(i)]
\item $H$ is a subring of $\QQ$ if and only if all finite entries of $\lgcd(H)$ are nonnegative.
\item $H$ is a unital subring of $\QQ$ if and only if all finite entries of $\lgcd(H)$ are zero. That is, if $H$ is a localization~$H=\ZZ[P^{-1}]$, for some $P\subseteq \PP$. \qed
 \end{enumerate}
\end{cor}

\begin{exm}[Subrings of $(\QQ,+,\cdot)$]
The previous corollary distinguishes three different situations:
 \begin{enumerate}[(a)]
\item The subgroup $\frac{1}{5}\,\ZZ[\frac{1}{6}]$ is not a subring of $\QQ$, since $\lgcd(\frac{1}{5}\,\ZZ[\frac{1}{6}]) = (-\infty,-\infty,-1)$ and $-1\not\in \NN\cup \{-\infty\}$; in fact, $(\frac{1}{5})^2 =\frac{1}{25}\notin \frac{1}{5}\,\ZZ[\frac{1}{6}]$.
\item The subgroup $5\,\ZZ[\frac{1}{6}]$ is a subring of $\QQ$ but non-unital, since $\lgcd(5\,\ZZ[\frac{1}{6}])=(-\infty,-\infty,1)$; in fact, $1\notin 5\,\ZZ[\frac{1}{6}]$.
\item The subgroup $\ZZ[\frac{1}{6}]$ is a unital subring of $\QQ$, since $\lgcd(\ZZ[\frac{1}{6}])=(-\infty, -\infty)$.
\end{enumerate}
\end{exm}

\begin{lem} 
Let $P\subseteq \PP$, and let $H$ be a nontrivial subgroup of $(\QQ,+)$. Then:
 \begin{enumerate}[(i)]
\item $H$ is a subgroup of $\ZZ[P^{-1}]$ if and only if $\lgcd_p(H) \in \NN$, for every $p \in \PP \setmin P$. 
\item $H$ is a subring of $\ZZ[P^{-1}]$ if  and only if there exists some $T\subseteq P$ such that
\[
\lgcd_p(H) \,=\, \bigg\{ \begin{array}{ll} -\infty & \text{if } p\in T \\ \in \NN  & \text{if } p\in \PP \setmin T \,. \end{array}
 \]
\item $H$ is a unital subring of $\ZZ[P^{-1}]$ if and only if there exists some $T\subseteq P$ such that $H=\ZZ[T^{-1}]$.
\item $H$ is an ideal of $\ZZ[P^{-1}]$ if and only if
\[
\lgcd_p(H) \,=\, \bigg\{ \begin{array}{ll} -\infty & \text{if } p\in P \\ \in \NN  & \text{if } p\in \PP \setmin P \,. \tag*{\qed}\end{array}
 \]
 \end{enumerate}
\end{lem}

\section{Homomorphisms and isomorphisms}

The behavior of homomorphisms involving subgroups of $\QQ$ is particularly simple and it essentially reduces to the crucial lemma below.

\begin{lem}\label{lem: Hom}
Let $H$ be a subgroup of $(\QQ,+)$. Then, every group homomorphism $H\to \QQ$ is of the form $\mu_r \colon x\mapsto rx$, where $r\in \QQ$. In particular, every nontrivial homomorphism $H\to \QQ$ is injective.
\end{lem}

\begin{proof}
If $H=\set{0}$ the only homomorphism is $0\mapsto 0$, which is of the claimed form with $r=0$. Otherwise, fix $s\in H\setmin \set{0}$ and take any $x\in H$. Since $s,x\in \QQ$, and $s\neq 0$, $\frac{x}{s} \in \QQ$, that is there exist $a\in \ZZ$ and $b\in \ZZ \setmin \set{0}$ such that $as=bx\in H$. Now, if $\varphi \colon H\to \QQ$ is a group homomorphism, then applying it to both sides we get $a\varphi(s)=\varphi(as)=\varphi(bx)=b\varphi(x)$. That is,
 $$
\varphi(x) \,=\, \frac{a}{b} \,\varphi(s) \,=\, \frac{\varphi(s)}{s} \, x\,, 
 $$
where $\frac{\varphi(s)}{s}\in \QQ$, as claimed. Injectivity is immediate since, for every $r\in \QQ \setmin \set{0}$, $rx=0$ implies $x=0$.
\end{proof}

In particular, every homomorphism from a subgroup of $\QQ$, to $\QQ$, extends to an endomorphism of $\QQ$. 

\begin{cor} \label{cor: End(Q)}
Let $H\neq \{0\}$ be a subgroup of $(\QQ,+)$. Then, the maps 
 $$
 \begin{array}{rcccl}
\QQ & \to & \End(\QQ) & \to & \Hom(H,\QQ) \\ r & \mapsto & \mu_r & \mapsto & \mu_{r\mid H} \,,
 \end{array}
 $$
are bijective, the first being a ring isomorphism from $(\QQ,+,\cdot)$ to $(\End(\QQ), +,\circ)$. \qed
\end{cor}

Hence, to characterize the endomorphisms (\resp automorphisms) of a subgroup $H\leqslant (\QQ,+)$, it suffices to consider endomorphisms $\varphi \in \End(\QQ,+)$ and impose the condition $\varphi(H)\leqslant H$ (\resp $\varphi(H)=H$). This can be done smoothly using our logarithmic description.

\begin{prop}
Let $H$ be a nontrivial subgroup of $(\QQ,+)$.
\begin{enumerate}[(i)]
\item \label{eq: End(H)} The set of endomorphisms of $H$ is
 \begin{align*}
\End(H) & =\set{ \mu_{r\mid H} \st r\in\QQ,\ rH\subseteq H} \\ & =\set{\mu_{0\mid H}} \cup
\set{\, \mu_{r\mid H} \st r\in\QQ^*,\ \plog_p(|r|)\geq 0 \ \forall p\in \PP \setminus \supp_{-\infty}(H)}.
 \end{align*}
Moreover, the map $\mu_{r\mid H} \mapsto r$ is an isomorphism of multiplicative monoids between 
$\End(H)$ and $\set{r\in\QQ\mid rH\subseteq H}$.
\item \label{eq: Aut(H)} The set of automorphisms of $H$ is
 \begin{align}
\Aut(H) & =\set{ \mu_{r\mid H} \st r\in\QQ^*,\ rH=H} \\ & =\set{ \, \mu_{r\mid H} \st r\in\QQ^*,\ \plog_p(|r|)=0 \ \forall p\in \PP \setminus \supp_{-\infty}(H) }.
 \end{align}
Moreover, the map $\mu_{r\mid H} \mapsto r$ is an isomorphism of groups between $\Aut(H)$ and
$\set{r\in\QQ^*\mid rH=H} \leqslant \QQ^*$. Hence,
 $$
\Aut(H) \isom \ZZ/2\ZZ \times 
\textstyle{\bigoplus_{\supp_{-\infty}(H)} \ZZ} \,.   
 $$
\end{enumerate}

\end{prop}

\begin{proof}
By \Cref{cor: End(Q)}, every endomorphism of $H$ is of the form $\mu_{r\mid H}$ for some $r\in \QQ$. Hence,
 \begin{align}
\varphi \in \End(H) \,\Biimp\, \varphi =\mu_{r|H}  \text{\, and \,} \mu_r(H) \leqslant H, \label{eq: muH in H} \\ \varphi \in \Aut(H) \,\Biimp\, \varphi =\mu_{r|H} \text{\, and \,} \mu_r(H)=H.
 \end{align} 
If $r=0$, then $\mu_{0\mid H}\in \End(H)$ and clearly $\mu_{0\mid H}\notin \Aut(H)$; so, in what follows we may assume $r\neq 0$. Then, using \Cref{eq: subgroup inclusion}, we have:
 \begin{align}
\mu_r(H) \leqslant H & \,\Biimp\, rH\leqslant H \notag \\
& \,\Biimp\, |r|H\leqslant H \notag \\
& \,\Biimp\, \lgcd(|r|H) \succeq \lgcd(H)\notag \\
& \,\Biimp\, \plog(|r|)+\lgcd(H) \succeq \lgcd(H) \notag \\
& \,\Biimp\, \plog_p(|r|)+\lgcd_p(H) \geq \lgcd_p(H), \quad \forall p\in \PP. \label{eq: log + lgcd}
 \end{align}
Now, we distinguish two cases: if $\lgcd_{p}(H)=-\infty$ then~\eqref{eq: log + lgcd} is always true and it imposes no condition on~$\plog_p(|r|)$; otherwise, $\lgcd_{p}(H)>-\infty$ and then~\eqref{eq: log + lgcd} is true if and only if $\plog_p(|r|) \geq 0$. Putting these conditions together we obtain \ref{eq: End(H)}. 

Finally, to see~\ref{eq: Aut(H)} it is enough to observe that
 $$
\mu_{r\mid H} \in \Aut(H) \,\Biimp\, \mu_{r\mid H}\in \End(H) \text{\, and \,} \mu_{r^{-1}\mid H}\in \End(H).
 $$
Applying~\ref{eq: End(H)} to both $r$ and $r^{-1}$, we get $\plog_p(|r|)\geq 0$ and $\plog_p(|r^{-1}|)\geq 0$, for every $p\in \PP$ such that $\lgcd_p(H)\neq -\infty$. Since $\plog_p(|r^{-1}|)=-\plog_p(|r|)$, it follows that~$\plog_p(|r|)=0$ for every such $p$. This proves~\ref{eq: Aut(H)}.
\end{proof}

The previous proposition takes a particularly simple form when applied to finite localizations.

\begin{cor}
Let $n\in \NN_{\geq 2} $ with $\supp(n)=\{ p_1,\ldots ,p_k\}$. Then,
 \begin{align}
\End\!\left(\ZZ\!\left[\tfrac{1}{n}\right]\right) & =\{\mu_0\} \cup \Set{ \mu_{\pm r} \st \textstyle{r=\prod_{p\in\PP} p^{e_p} },\ e_p\in\ZZ,\ e_p\ge 0\ \forall p\notin \supp(n)}. \\[3pt] \Aut\!\left(\ZZ\!\left[\tfrac{1}{n}\right]\right)
& \,=\, \Set{ \mu_{\pm r} \st r=p_1^{e_1} \cdots p_k^{e_k} \ , \, e_1,\ldots,e_k \in \ZZ }\isom \ZZ/2\ZZ \oplus \ZZ^k.
 \end{align}
Thus, $\End(\ZZ[\frac{1}{n}])$ can be identified with the ring $\ZZ[\frac{1}{n}]$ itself, and $\Aut(\ZZ[\frac{1}{n}])$ with its group of units~$\ZZ[\frac{1}{n}]^\times \isom \ZZ/2\ZZ \oplus \ZZ^k$. \qed
\end{cor}

As a consequence of \Cref{lem: Hom}, we obtain a classification of isomorphic classes within subgroups of $\QQ$. Recall that two subgroups $H,K$ of a given group are called \defin{commensurable} if their intersection $H\cap K$ has finite index in both $H$ and $K$.

\begin{prop}\label{prop: isom}
Let $H,K$ be nontrivial subgroups of $(\QQ,+)$. Then, the following conditions are equivalent:
 \begin{enumerate}[(a)]
\item $H\isom K$; \label{item: isom}
\item $K=rH$, for some $r\in \QQ^*$; \label{item: K = rH}
\item $\lgcd(H)-\lgcd(K) \in \Seq_0(\ZZ)$; \label{item: lgcd - lgcd}
\item $\dist(H,K) < + \infty$; \label{item: dist}
\item $H$ and $K$ are commensurable. \label{item: commensurable}
 \end{enumerate}
\end{prop}

\begin{proof}
$[\ref{item: isom}\Rightarrow\ref{item: K = rH}]$
Let $\varphi \colon H\to K$ be an isomorphism. Since $K\leqslant \QQ$ we may regard $\varphi$ as a group homomorphism $\varphi\colon H\to \QQ$. By \Cref{lem: Hom}, we know that $\varphi =\mu_{r\mid H}$ for some $r\in \QQ$. Since $\varphi$ is injective, $r \neq 0$. That is, $K=\mu_r(H)=rH$ for some $r\in \QQ^*$.

$[\ref{item: K = rH}\Rightarrow\ref{item: lgcd - lgcd}]$ If $K=rH$ then $K=|r|H$ and 
 \[
\lgcd(K)=\lgcd(|r|H)=\plog(|r|)+\lgcd(H).
 \]
Since $|r|\in \QQ^+$, $\plog(|r|) \in \Seq_0(\ZZ)$, the sequences $\lgcd(H)$ and $\lgcd(K)$ differ at only finitely many coordinates, and therefore $\lgcd(H)-\lgcd(K)\in\Seq_0(\ZZ)$.

$[\ref{item: lgcd - lgcd}\Rightarrow\ref{item: dist}]$ 
If $\lgcd(H) - \lgcd(K) = \s{s} \in \Seq_0(\ZZ)$ then
 \begin{equation}\label{eq: finite sum}
\dist(H,K) \,=\, \sum_{p\in\PP} \dist(\lgcd_p(H),\lgcd_p(K)) \,=\, \sum_{p\in\PP} |\s{s}(p)| \,.
 \end{equation}
Since $\supp(s)$ is finite and $\s{s}(p) \in \ZZ$ for every $p \in \supp(\s{s})$, the sum in \eqref{eq: finite sum} is finite and $\dist(H,K)<+\infty$, as claimed.

$[\ref{item: dist}\Rightarrow\ref{item: commensurable}]$
Note that if $\dist(H,K)<+\infty$ then 
 \[
\max(\lgcd(H),\lgcd(K)) - \lgcd(H) \in \Seq_0(\NN)
 \]
since this sequence is coordinatewise nonnegative and its support is contained in $\set{p\in\PP\st \lgcd_p(H)\neq\lgcd_p(K)}$, which is finite because $\dist(H,K)<+\infty$. Then using the expressions \eqref{eq: lgcd cap} for the intersection and \eqref{eq: index} for the index we have
 \begin{align*}
|H:H \cap K| 
&\,=\, \PP^{\max(\lgcd(H),\lgcd(K))-\lgcd(H)}<+\infty \,,
 \end{align*}
and similarly for $|K:H\cap K|$. Hence $H$ and $K$ are commensurable, as claimed.

$[\ref{item: commensurable}\Rightarrow\ref{item: isom}]$
Let $L=H\cap K$, and write $\s{h}=\lgcd(H)$, $\s{k}=\lgcd(K)$, and $\s{l}=\lgcd(L)$. Since $H$ and $K$ are commensurable, from the index formula \eqref{eq: index} we have that $\s{l}-\s{h}, \s{l}-\s{k}\in \Seq_0(\NN)$. Substracting, we obtain $\s{h}-\s{k}=\s{r}\in \Seq_0(\ZZ)$. That is, $H=rK$, with $r=\PP^{\,\s{r}}\in \QQ^*$ and hence $H$ and $K$ are isomorphic. This completes the proof.
\end{proof}

\begin{cor}
Every nontrivial subgroup $H\leqslant (\QQ,+)$ is isomorphic to a subgroup $K\leqslant (\QQ,+)$ such that $\lgcd(K) \preceq \s{0}$, that is, with $\ZZ\leqslant K$. \qed
\end{cor}

The previous characterization shows that the isomorphism type of a subgroup $H\leqslant \QQ$ is determined by the tail behavior of the sequence $\lgcd(H)$, that is, by its class modulo $\Seq_0(\ZZ)$. This can also be seen at the level of natural presentations.

Let $\s{s}\in\Seq_{\leq 0}(\Zb)$, and consider the corresponding subgroup $H_{\s{s}} =\pm\PP^{\,\fSucc\mathbf{s}}\cup\{0\}\leqslant \QQ$. Choose a sequence
 $$
\mathbf{e_0} \succeq \mathbf{e_1} \succeq \mathbf{e_2} \succeq \cdots
 $$
in $\fSucc\mathbf{s}$ which is coinitial (\ie such that $\forall \mathbf{e} \in \fSucc \mathbf{s}\ \exists i\in \NN \ \text{such that}\ \s{e_i}\preceq \mathbf{e}$). Equivalently, the associated cyclic subgroups form an increasing chain
 $$
\gen{\PP^{\,\mathbf{e_0}}} \leqslant \gen{\PP^{\,\mathbf{e_1}}} \leqslant \gen{\PP^{\,\mathbf{e_2}}} \leqslant \cdots \leqslant H_{\s{s}}
 $$
whose union is the whole $H_{\mathbf{s}}$. 
For each $i\in \NN$, consider the natural numbers $m_i =\PP^{\,\s{e_i}- \s{e_{i+1}}}\in\NN^+$.

\begin{thm} \label{thm: H pres}
Let $\s{s}\in\Seq_{\leq 0}(\Zb)$. With the above notation, the subgroup $H_{\s{s}}$ admits the presentation
 \begin{equation}\label{eq: pres Hs}
H_{\mathbf{s}} \cong \pres{x_i \ (i \in \NN)}{x_i= x_{i+1}^{m_i}\ (i\in \NN)}.
 \end{equation}
\end{thm}

\begin{proof}
Consider the abstract group $G$ presented as in~\eqref{eq: pres Hs} and define a homomorphism $\Phi\colon G\to H_{\mathbf{s}}$, $x_i\mapsto \PP^{\mathbf{\,e_i}}$. This is well defined because, for every $i\in\NN$, $m_i\PP^{\s{e_{i+1}}}=\PP^{\s{e_i}-\s{e_{i+1}}}\PP^{\s{e_{i+1}}} =\PP^{\s{e_i}}$, and the relation $x_i=x_{i+1}^{m_i}$ is satisfied in $H_{\mathbf{s}}\leqslant \QQ$.

Let us prove that $\Phi$ is surjective. Since the sequence $(\s{e_i})_{i\in\NN}$ is coinitial in $\fSucc\mathbf{s}$, for every $\mathbf{e}\in\fSucc\mathbf{s}$ there exists $i\in\NN$ such that $\mathbf{e_i}\preceq\mathbf{e}$. Hence, $\mathbf{e}-\mathbf{e_i}\in\Seq_0(\NN)$,
and so $\PP^{\mathbf{e}} =\PP^{\mathbf{e}-\mathbf{e_i}}\PP^{\mathbf{e_i}}\in \gen{ \PP^{\mathbf{e_i}} }\leqslant \im\Phi$. Since $H_{\mathbf{s}}=\pm\PP^{\,\fSucc\mathbf{s}}\cup\{0\}$, we conclude that $\Phi$ is onto.

It remains to prove that $\Phi$ is injective. For $i<j$, repeated use of the defining relations gives
 $$
x_i=x_j^{m_i m_{i+1}\cdots m_{j-1}}.
 $$
Therefore, every element of $G$ can be written as a power of some generator $x_j$, with $j$ big enough. Now suppose that $g\in\ker\Phi$. Write $g=x_j^z$. Then $0=\Phi(g)=z\,\PP^{\s{e_j}}\in\QQ$. Since $\PP^{\s{e_j}}\neq0$, it follows that $z=0$. Hence $g=\trivial$. Thus $\ker\Phi=\trivial$, and $\Phi$ is injective.

Therefore $\Phi$ is an isomorphism, and the claimed result follows.
\end{proof}

Note that, in the proof above, $G$ is an abelian group although no commutativity relations are needed in the multiplicative presentation~\eqref{eq: pres Hs}: indeed, if $i<j$, then, as seen in the previous proof, repeated use of the defining relations gives $x_i=x_j^{m_i m_{i+1}\cdots m_{j-1}}$; therefore, $x_i$ and $x_j$ do commute. 

\begin{rem}
\Cref{thm: H pres} generalizes the well-known presentation
\[\pres{x_i \ (i \in \NN)}{x_i= x_{i+1}^{i+1}\ (i\in \NN)}
\]
of the additive rational numbers (see, for example, \cite[Chapter 5, Section 7]{johnsonPresentationsGroups1997})  to arbitrary subgroups of $(\QQ,+)$. Moreover, the same pattern applies to the presentations above for an arbitrary subgroup $H\leqslant\QQ$, with the successive factors $i+1$ replaced by those determined by the chosen coinitial sequence for $H$.
\end{rem}

Let us remark that $\Seq(\ZZ) =\ZZ^\PP =\prod_{p\in \PP} \ZZ$ is a $\ZZ$-module, that $\Seq_0(\ZZ) =\bigoplus_{p\in \PP} \ZZ$ is a free $\ZZ$-submodule, and hence the quotient $\Seq(\ZZ)/\Seq_0(\ZZ)$ is well defined. Moreover, this quotient construction restricts to $\Seq_{\leq0}(\ZZ)/\Seq_0(\ZZ)$, and further extends to $\Seq_{\leq0}(\Zb) / \Seq_0(\ZZ)$ in the natural way.  

\begin{thm}
The family of isomorphic classes of nontrivial subgroups of $(\QQ,+)$ is in bijection with $\Seq_{\leq0}(\Zb) / \Seq_0(\ZZ)$. In particular, there are uncountably many classes of nonisomorphic subgroups in $(\QQ,+)$.
\end{thm}

\begin{proof}
    The first claim is just a reformulation of \Cref{prop: isom}. For the second claim, it is enough to realize that
    all the elements in the uncountable family $\set{{-\infty \cdot \mathbf{1}_S} \st S \subseteq \PP}$ are pairwise nonisomorphic since they have different ($-\infty$)-supports.
\end{proof}

As a particular case of \Cref{prop: isom}, we obtain a straightforward characterization of isomorphisms between subgroups of $(\QQ,+)$ with finite support. 

\begin{rem}
    Note that a subgroup $H$ of $(\QQ,+)$ has finite support if and only if it is of one of the following forms:
\begin{enumerate}[(i)]
    \item $H = r \,\ZZ$, with $r \in \QQ^+$;
    \item $H = r\, \ZZ\!\left[\frac{1}{n}\right]$, with $r \in \QQ^+$ and $n \in \NN_{\geq 2}$. 
\end{enumerate}
\end{rem}

\begin{cor}
Let $n,m \in \NN_{\geq 2}$ and $r,s \in \QQ^{+}$. Then
\[
r\,\ZZ\!\left[\tfrac{1}{n}\right] \isom s\,\ZZ\!\left[\tfrac{1}{m}\right]
\;\Leftrightarrow\;
\supp(n) = \supp(m).
\]
Moreover,
\[
\ZZ\!\left[\tfrac{1}{n}\right] \isom \ZZ\!\left[\tfrac{1}{m}\right]
\;\Leftrightarrow\;
\ZZ\!\left[\tfrac{1}{n}\right] = \ZZ\!\left[\tfrac{1}{m}\right]. \tag*{\qed}
\]
\end{cor}






\begin{rem}
It is clear from the previous results, that the algorithmic landscape for subgroups of $(\QQ,+)$ is sharply polarized. In the usual setting of algorithmic group theory, where the input subgroups are finitely generated, every subgroup of $\QQ$ is cyclic. Consequently, isomorphism, membership, and subgroup inclusion reduce to elementary computations with rational numbers (or, equivalently, integer sequences with finite support) and are therefore essentially trivial. For arbitrary, non-(finitely generated) subgroups, on the other side, no uniform algorithmic theory can be expected. Indeed, most subgroups of $\QQ$ are not computable, and hence cannot even be presented as admissible inputs for a decision procedure. 
\end{rem}

\section{Intersection configurations}\label{sec: intersections}

In this section, we shall use our logarithmic description to unveil the behavior of intersections of subgroups of $(\QQ,+)$ with respect to rank and finite generation.

The study of finite generation of subgroup intersections goes back to Howson~\cite{howsonIntersectionFinitelyGenerated1954}, who proved that the intersection of any two finitely generated subgroups of a free group is again finitely generated. A convenient geometric and algorithmic framework for studying this and related facts is the theory of Stallings automata (see~\cite{stallingsTopologyFiniteGraphs1983} and~\cite{delgadoStallingsAutomata2024}).
Beyond free groups, this closure condition came to be studied as a group-theoretic property in its own right, now usually called the \emph{Howson property} or \emph{finitely generated intersection property (f.g.i.p.)}.


\begin{defn}
A group is said to satisfy the \defin{Howson property} (to be \defin{Howson}) if the intersection of any two (and hence finitely many) finitely generated subgroups is again finitely generated.
\end{defn}

As all abelian groups, $(\QQ,+)$ and all its subgroups are obviously Howson. However, this property can be refined in two different directions: 
\begin{enumerate*}[(i)]
\item by considering finite multiple intersections, as introduced in~\cite{delgadoIntersectionConfigurationsFree2024} through the notion of intersection configuration;
\item and by recording the specific ranks of the intersections instead of just their finite generation. 
\end{enumerate*}
Putting both refinements together we reach the notion of \defin{rank intersection configuration}, defined below.

Let $\bN =\NN\cup\set{\infty}$. For $k\in \NN$, let $[k]=\set{i\in \NN \st 1\leq i\leq k}$, and let $\Pk =2^{[k]}\setmin \set{\varnothing}$, regarded as an ordered set under inclusion. 


\begin{defn}
A \defin{rank intersection configuration of size} $k\geq 1$ is a map
$\chi \colon \Pk \to \bN$. If we collapse all the finite values to a single one, called $\mathsf{F}$, we recover the notion of \defin{binary intersection configuration} $\widetilde{\chi}\colon \Pk \to \set{\mathsf{F},\infty}$, introduced in~\cite{delgadoIntersectionConfigurationsFree2024}.\footnote{In~\cite{delgadoIntersectionConfigurationsFree2024}, finitely generated groups are represented by $0$ and non-(finitely generated) groups by $1$; we change notation to $\mathsf{F}$ and $\infty$ to avoid confusion with ranks.} The generic term \defin{intersection configuration} refers to either type.
\end{defn}

\begin{defn}
An intersection configuration $\chi$ of size $k$ is \defin{realizable} in a group~$G$ if there exist subgroups $H_1,\ldots,H_k \leqslant G$ such that $\rk(\bigcap_{i\in I} H_i)=\chi (I)$, for every $I\in\Pk$.\footnote{For binary configurations, we use the convention $\rk(H)=\mathsf{F}$ to mean that $\rk(H)<\infty$.}
\end{defn}

We shall consider $\bN$ with its natural order, and $\set{\mathsf{F},\infty}$ with the order inherited from $\bN$, namely $\mathsf{F}<\infty$.

\begin{defn}
For any $r$ in the codomain of an intersection configuration $\chi$, we say that $I\in \Pk$ is \defin{$r$-maximal} (\resp \defin{$r$-minimal}) for $\chi$ if it is a $\subseteq$-maximal (\resp $\subseteq$-minimal) element of $\chi^{-1}(r) = \set{J \st \chi(J) = r}$. We denote by $\mathcal{M}_r(\chi)$ the set of $r$-maximal subsets for $\chi$.
\end{defn}

A convenient graphical representation of intersection configurations is through labeled hypercubes. In our hypercube diagrams, inclusion of indices will be directed downwards; thus, $r$-maximal subsets are bottom $r$-sets, and $r$-minimal subsets are top $r$-sets. Moreover, we use black ($\infty$) and white ($\mathsf{F}$) to denote infinite and finite generation, respectively. For example, if $\chi$ is the configuration depicted in~\Cref{fig:cube}, then $\mathcal{M}_\infty(\chi)=\set{\set{1,2}, \set{2,3}}$.

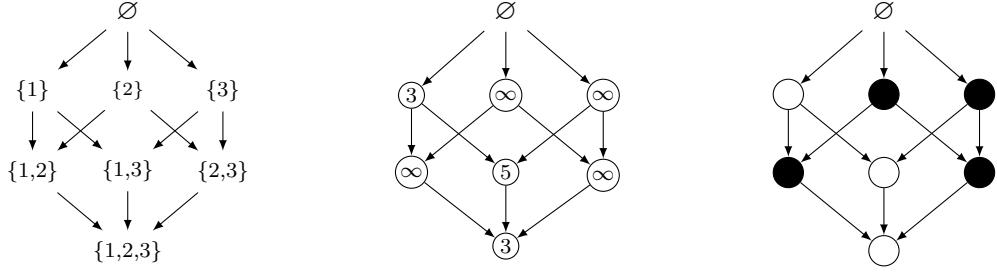
\begin{figure}[H]
\centering
\begin{tikzpicture}[
    >=latex,
    state/.style={circle,draw,minimum size=2mm,inner sep=1.1pt},
    finite/.style={circle,draw,fill=white,minimum size=4mm,inner sep=0pt},
    infinite/.style={circle,draw,fill=black,minimum size=4mm,inner sep=0pt}
]

\begin{scope}
    \newcommand{\dx}{0.55}
    \newcommand{\dy}{0.52}

    \node (000) {$\varnothing$};

    \node (010) [below=\dy of 000] {\small$\scriptstyle\{2\}$};
    \node (100) [left=\dx of 010] {$\scriptstyle\{1\}$};
    \node (001) [right=\dx of 010] {$\scriptstyle\{3\}$};

    \node (110) [below=\dy of 100] {$\scriptstyle\{1,2\}$};
    \node (101) [below=\dy of 010] {$\scriptstyle\{1,3\}$};
    \node (011) [below=\dy of 001] {$\scriptstyle\{2,3\}$};

    \node (111) [below=\dy of 101] {$\scriptstyle\{1,2,3\}$};

    \path[->] (000) edge (100);
    \path[->] (000) edge (010);
    \path[->] (000) edge (001);

    \path[->] (100) edge (110);
    \path[->] (100) edge (101);
    \path[->] (010) edge (110);
    \path[->] (010) edge (011);
    \path[->] (001) edge (101);
    \path[->] (001) edge (011);

    \path[->] (110) edge (111);
    \path[->] (101) edge (111);
    \path[->] (011) edge (111);
\end{scope}

\begin{scope}[shift={(5,0)}]
    \newcommand{\dx}{0.85}
    \newcommand{\dy}{0.62}

    \node (m000) {$\varnothing$};

    \node[state] (m010) [below=\dy of m000] {$\scriptstyle\infty$};
    \node[state] (m100) [left=\dx of m010] {$\scriptstyle 3$};
    \node[state] (m001) [right=\dx of m010] {$\scriptstyle\infty$};

    \node[state] (m110) [below=\dy of m100] {$\scriptstyle\infty$};
    \node[state] (m101) [below=\dy of m010] {$\scriptstyle 5$};
    \node[state] (m011) [below=\dy of m001] {$\scriptstyle\infty$};

    \node[state] (m111) [below=\dy of m101] {$\scriptstyle 3$};

    \path[->] (m000) edge (m100);
    \path[->] (m000) edge (m010);
    \path[->] (m000) edge (m001);

    \path[->] (m100) edge (m110);
    \path[->] (m100) edge (m101);
    \path[->] (m010) edge (m110);
    \path[->] (m010) edge (m011);
    \path[->] (m001) edge (m101);
    \path[->] (m001) edge (m011);

    \path[->] (m110) edge (m111);
    \path[->] (m101) edge (m111);
    \path[->] (m011) edge (m111);
\end{scope}

\begin{scope}[shift={(10,0)}]
    \newcommand{\dx}{0.85}
    \newcommand{\dy}{0.62}

    \node (r000) {$\varnothing$};

    \node[infinite] (r010) [below=\dy of r000] {};
    \node[finite]   (r100) [left=\dx of r010] {};
    \node[infinite] (r001) [right=\dx of r010] {};

    \node[infinite] (r110) [below=\dy of r100] {};
    \node[finite]   (r101) [below=\dy of r010] {};
    \node[infinite] (r011) [below=\dy of r001] {};

    \node[finite]   (r111) [below=\dy of r101] {};

    \path[->] (r000) edge (r100);
    \path[->] (r000) edge (r010);
    \path[->] (r000) edge (r001);

    \path[->] (r100) edge (r110);
    \path[->] (r100) edge (r101);
    \path[->] (r010) edge (r110);
    \path[->] (r010) edge (r011);
    \path[->] (r001) edge (r101);
    \path[->] (r001) edge (r011);

    \path[->] (r110) edge (r111);
    \path[->] (r101) edge (r111);
    \path[->] (r011) edge (r111);
\end{scope}

\end{tikzpicture}
\caption{The lattice of subsets of $[3]$ on the left, a particular rank intersection configuration of size $3$ in the middle, and the associated binary configuration on the right, where black vertices denote infinitely generated intersections and white vertices denote finitely generated intersections}
\label{fig:cube}
\end{figure}

\begin{rem}
Note that, using this representation, an $r$-minimal (\resp $r$-maximal) subset of $\chi$ corresponds to an $r$-node in the hypercube which has no $r$-predecessors (\resp no $r$-sucessors); that is, no $r$-nodes above (below) it. 
\end{rem}

For $I\in \Pk$, we abbreviate $H_I :=\bigcap_{i \in I} H_i$. Observe that $H_{I\cup J}=H_I\cap H_J$ and hence, in the hypercube diagrams, index inclusions are directed downwards whereas subgroup inclusions are directed upwards.

\begin{defn}
An intersection configuration $\chi$ is said to be
 \begin{enumerate}[(i)]
\item \defin{decreasing} if, for all $I,J\in\Pk$, $I\subseteq J$ implies $\chi(I)\geq \chi(J)$ (\ie the values of $\chi$ decrease as we go down through the hypercube representation).
\item \defin{Howson} if, for every $I,J\in\Pk$, $\chi(I)\neq\infty$ and $\chi(J)\neq\infty$ imply $\chi(I\cup J)\neq\infty$.
 \end{enumerate}
\end{defn}

Note that a binary configuration is decreasing if and only if its corresponding hypercube diagram does not contain the left pattern from~\Cref{fig:forbidden-patterns}, and \emph{Howson} if and only if it does not contain the right pattern (the snaked arrows denote directed paths and the black vertex is where both paths first meet). 

\begin{figure}[H]
\centering
\begin{tikzpicture}[
    >=latex,
    state/.style={circle,draw,minimum size=4mm},
    decoration={
        snake,
        segment length=2mm,
        amplitude=0.5mm,
        post length=1.5mm
    }
]

\begin{scope}[shift={(0,0)}]
    \node[state] (d0) {};
    \node[state,fill=black] (d1) [below=0.5 of d0] {};

    \path[->] (d0) edge (d1);

\end{scope}

\begin{scope}[shift={(4.5,0)}]
    \node (h1) {};
    \node[state] (h10) [left=1 of h1] {};
    \node[state] (h01) [right=1 of h1] {};
    \node[state,fill=black] (h11) [below=0.5 of h1] {};

    \path[->] (h10) edge[decorate] (h11);
    \path[->] (h01) edge[decorate] (h11);

\end{scope}

\end{tikzpicture}
\caption{Obstructions to decreasing (left) and Howson (right) binary configurations}
\label{fig:forbidden-patterns}
\end{figure}
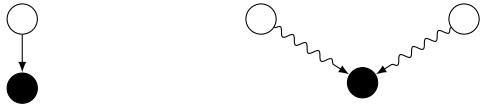
Of course, every decreasing binary configuration is Howson (and not viceversa), since every Howson obstruction contains a decreasing obstruction. For example, the configuration(s) in \Cref{fig:cube} is Howson but not decreasing.
\medskip

The goal of this last section is to give an explicit characterization of the intersection configurations that are realizable in an arbitrary given subgroup $G\leqslant \QQ$. Since every subgroup of a finitely generated abelian group is finitely generated, with rank at most that of the ambient group, the result below follows.

\begin{lem} \label{lem: decreasing}
Every intersection configuration realizable in an abelian group is decreasing; in particular, it is Howson. \qed
\end{lem}





\medskip

Combining the logarithmic description of subgroup intersections in $(\QQ,+)$ with the criterion for finite generation, we can give an explicit characterization of the (both binary and rank) intersection configurations realizable in any given subgroup $G\leqslant \QQ$. First we need a couple of lemmas. 

\begin{lem} \label{lem: 2 props}
Let $H,K$ be subgroups of $(\QQ,+)$.
 \begin{enumerate}[(i)]
\item If $H\leqslant K$ then $\supp_{-\infty}(H) \subseteq \supp_{-\infty} (K)$.
\item \label{item: infty cap}
$\supp_{-\infty} (H\cap K)=\supp_{-\infty} (H)\cap \supp_{-\infty} (K)$.
 \end{enumerate} 
\end{lem}

\begin{proof}
If $H$ is a subgroup of $K$, then $\lgcd_p(K)\leqslant \lgcd_p(H)$
 for every $p\in\PP$. Hence, if $\lgcd_p(H)=-\infty$, then necessarily
$\lgcd_p(K)=-\infty$, proving the first claim. For the second claim, recall that $\lgcd(H\cap K)=\max(\lgcd(H),\lgcd(K))$. Since $\max(a,b)=-\infty$ if and only if $a=b=-\infty$, the claim follows.
\end{proof}

\begin{lem} \label{lem: finite support}
Let $H\leqslant K\leqslant (\QQ,+)$. If $K$ has finite support then so does $H$. 
\end{lem}

\begin{proof}
Let $\s{h}=\lgcd(H)$ and $\s{k}=\lgcd(K)$. By~\eqref{eq: subgroup inclusion}, we have $\s{h}(p) \geq \s{k}(p)$, for every $p\in \PP$.
Then, it is enough to combine this with the following two facts:
 \begin{enumerate}[(i)]
\item by~\Cref{lem: lgcd eventually negative}, there exists $p_H \in\PP$ such that $\s{h}(p)\leq 0$, for every $p\geq p_H$; 
\item by hypotheses, there exists $p_K \in \PP$ such that $\s{k}(p)=0$, for every $p\geq p_K$.
 \end{enumerate}
Hence, if $\PP\ni p\geq\max(p_H,p_K)$, then $0\geq\s{h}(p) \geq\s{k}(p)=0$; therefore $\s{h}(p)=0$, and $H$ has finite support, as claimed.
\end{proof}

\begin{rem} \label{rem: fg iff no infty}
If $H\leqslant (\QQ,+)$ has finite support, then $H$ is finitely generated if and only if $\supp_{-\infty}(H) = \varnothing$. This follows easily from~\Cref{cor: fg iff}.
\end{rem}

\begin{rem}\label{rem: nontrivial}
The intersection of two (and so, finitely many) nontrivial subgroups of $(\QQ,+)$ is always nontrivial; this follows immediately from the logarithmic description of intersections, $\lgcd(H\cap K)=\max(\lgcd(H),\lgcd(K))$; see \eqref{eq: lgcd cap}. Note that this is a very special property of the additive group of rationals $(\mathbb{Q}. +)$.
\end{rem}


\begin{thm}\label{thm: main}
Let $G$ be a nontrivial subgroup of $(\QQ,+)$ and let $\chi \colon \Pk \to \bN$ be a rank intersection configuration. Then, $\chi$ is realizable in $G$ if and only if
 \begin{enumerate}[(i)]
\item\label{item: im rho} $\im(\chi) \subseteq \set{0,1,\infty}$,
\item\label{item: rho decreasing} $\chi$ is decreasing,
\item\label{item: 0 minimal} every $0$-minimal subset of $\chi$ is a singleton,
\item\label{item: max bound} if $G$ has finite support then $|\mathcal{M}_{\infty}(\chi)|\leq |\supp_{-\infty}(G)|$.
\end{enumerate}
\end{thm}  

\begin{proof}

The necessity of conditions is clear: \ref{item: im rho} follows from \Cref{rem: subgroup ranks}. \ref{item: rho decreasing} follows from~\Cref{lem: decreasing}. \ref{item: 0 minimal} follows from~\Cref{rem: nontrivial}. Finally, to see~\ref{item: max bound}, assume that $G$ has finite support. Then, for any $H,K\leqslant G$, \Cref{lem: finite support} tells us that $H\cap K$ has finite support, and \Cref{rem: fg iff no infty} together with~\Cref{lem: 2 props}.\ref{item: infty cap} that $H\cap K$ is finitely generated if and only if
 $$
\supp_{-\infty}(H) \cap \supp_{-\infty}(K) = \supp_{- \infty}(H \cap K) = \varnothing \,.
 $$
Now suppose that $H_1,\ldots,H_k \leqslant G$ realize $\chi$, and take $I,J\in \mathcal{M}_{\infty}(\chi)$, $I\neq J$. This means that $H_I$ and $H_J$ are not finitely generated, but $H_I \cap H_J =H_{I\cup J}$ is finitely generated. That is, $\varnothing \neq \supp_{-\infty}(H_I)\subseteq \supp_{-\infty}(G)$ and $\varnothing \neq \supp_{-\infty}(H_J) \subseteq \supp_{-\infty}(G)$ but ${\supp_{-\infty}(H_I) \cap \supp_{-\infty}(H_J)=\varnothing}$. In other words, $\set{\supp_{-\infty}(H_I) \st I\in \mathcal{M}_{\infty}(\chi)}$ is a family of $|\mathcal{M}_{\infty}(\chi)|$ nonempty and pairwise disjoint subsets of $\supp_{-\infty}(G)$. By the pigeonhole principle, if $|\mathcal{M}_{\infty}(\chi)|>|\supp_{-\infty}(G)|$ there must be at least two subsets in $\mathcal{M}_{\infty}(\chi)$ which are not disjoint, contradicting the realizability of $\chi$. This completes the proof of the implication to the right.

\medskip

For the implication to the left, let $\s{g}=\lgcd(G)$, and let us distinguish two cases. 

\medskip

\textbf{Case 1: $G$ has infinite support}. Since $\supp_{>0}(G)$ is always finite, the negative support $N=\supp_{<0}(G)$ must be infinite as well. Let $\chi\colon\Pk\to \set{0,\trivial, \infty}$ be a rank configuration satisfying \ref{item: im rho}, \ref{item: rho decreasing}, and \ref{item: 0 minimal} (\ref{item: max bound} is vacuous in this case), and  let $\mathcal{M} =\mathcal M_\infty(\chi)$ be the set of $\infty$-maximal subsets for $\chi$. We are going to realize $\chi$ by distinguishing two subcases. 

\textbf{Subcase 1.1:} $\mathcal{M} =\varnothing$, that is, $\im(\chi) \subseteq \set{0,1}$, then the family of subgroups defined as
 \[
H_i =\begin{cases}
\set{0} & \text{if } \chi(\set{i}) = 0,\\
\gen{g_i} & \text{if } \chi(\set{i}) = 1,
\end{cases}
 \]
where $0\neq g_i\in G$ is arbitrary (for example, all of them equal), realizes $\chi$. In fact, for each $I\in \Pk$, if $\chi(I)=0 $ then, by~\ref{item: 0 minimal}, $\chi(\{i\})=0$ for some $i\in I$; so, $H_I =\set{0}$ and hence $\rk(H_I)=0=\chi(I)$. Otherwise, $\chi(I)=1 $ and, since $\im(\chi) \subseteq \set{0,1}$ and using~\ref{item: rho decreasing}, $\chi(\{i\})=1$ for every $i\in I$; so, $H_I$ is and intersection of finitely many nontrivial cyclic subgroups of $G$ and hence, nontrivial cyclic itself (see~\Cref{rem: nontrivial}); therefore, $\rk(H_I)=1=\chi(I)$. 

\textbf{Subcase 1.2:} $\mathcal{M}\neq\varnothing$, that is, $\infty\in \im(\chi)$. Since $\mathcal{M}$ is finite, while $N$ is infinite, we can choose a partition $N=\bigsqcup_{J\in \mathcal{M}} N_J$, with each $N_J$ being infinite, $J\in \mathcal{M}$. For each $i\in[k]$ with $\chi(\{i\})\neq 0$, define a sequence $\s{h_i}\in\Seq_{\leq0}(\Zb)$ by
 \[
\s{h_i}(p)= \begin{cases} \s{g}(p), & \text{if } p\in \supp_{>0}(\s{g}), \\ \s{g}(p), & \text{if } p\in \bigcup \set{N_J \st J\in \mathcal{M},\, i\in J}\subseteq N=\supp_{<0}(G), \\ 
 0, & \text{otherwise.} \end{cases}
 \]
And consider the subgroups $H_i =\{0\}$ if $\chi(\{i\})=0$, and $H_i =H_{\s{h_i}}\neq \{0\}$ otherwise. 
We claim that the family of subgroups $\set{H_1,\ldots,H_k}$ realizes $\chi$ in $G$. Indeed, by construction, $\s{h_i}\succeq \s{g}$ for every $i$ with $\chi(\{i\})\neq 0$, and hence $H_i\leqslant G$ for all $i$.

We first claim that, for every $I\in \Pk$,
 \begin{equation}\label{eq: nfg iff incl}
\rk(H_I)=\infty \ \Longleftrightarrow \chi(I) =\infty. 
 \end{equation}
Indeed, if $\chi(I)=0$ the claim follows from \ref{item: 0 minimal}. Otherwise, we have $\lgcd(H_I)=\max_{i\in I}\s{h_i}$.
Now, if $\chi(I)=\infty$ then $I\subseteq J$ for some $J\in\mathcal M$ and hence, for every $p\in N_J$ and every $i\in I$, 
$\s{h_i}(p)=\s{g}(p)<0$.
Thus, 
 $$
\lgcd_p(H_I)=\max_{i\in I}\set{\s{h_i}(p)}=\s{g}(p)<0 \qquad(\forall p\in N_J).
 $$
Since $N_J$ is infinite, $H_I$ has infinite support, and therefore $\rk(H_I) =\infty$.

Conversely, suppose that $\chi(I) \neq\infty$. Then, by \ref{item: rho decreasing}, $I$ is not contained in any $J\in\mathcal{M}$, i.e., for every $J\in\mathcal{M}$, there exists $i_J\in I\setminus J$. Hence, for every $p\in N_J$, $\s h_{i_J}(p)=0$. Since all values $\s{h_i}(p)$, $p\in N$, $i\in I$, are either $\s{g}(p)<0$ or $0$, we get
 \[
\lgcd_p(H_I)=\max_{i\in I}\set{\s{h_i}(p)}=0, \qquad(\forall p\in N = \textstyle{\bigsqcup_{J \in \mathcal{M}} N_J}).
 \]
Hence, $\supp(\lgcd(H_I))\subseteq \supp_{>0}(\s h)$, which is finite, and $H_I$ is finitely generated. This completes the proof of the claim.

To finish proving that the family of subgroups $H_1,\ldots,H_k$ realizes $\chi$ in $G$  it remains to observe that $\rk (H_I)=0$ if and only if $H_i=\{0\}$ for some $i\in I$, which happens if and only if $\chi(\{i\})=0$ (and so, $\rk (H_I)=1$ if and only if $\chi(I)=1$). 

\medskip

\textbf{Case 2: $G$ has finite support}. Let $\chi\colon\Pk\to \set{0,\trivial, \infty}$ be a rank configuration satisfying~\ref{item: im rho}, \ref{item: rho decreasing}, \ref{item: 0 minimal} and~\ref{item: max bound} (now crucially meaning $|\mathcal{M}|\leq |\supp_{-\infty}(G)|$), and let us realize $\chi$. Consider an injective map $\mathcal{M}\to \supp_{-\infty}(G)$, $J \mapsto p_J$ (which exists, since $|\mathcal{M}| \leq |\supp_{-\infty}(G)|$); that is, associate to every $\infty$-maximal subset $J$ for $\chi$, a different prime $p_J \in \supp_{-\infty}(G)$. For each $i\in [k]$ such that $\chi(\set{i})=0$, define $H_i =\set{0}$; otherwise (that is, if $\chi(\set{i}) \neq 0$), define $H_{i}=H_{\s{h_i}}\neq \set{0}$ where 
 \[
\s{h_i}(p)=\begin{cases} -\infty, & \text{if } p\in \set{p_J \st J\in\mathcal{M},\, i\in J}, \\[1mm] \max\set{0,\s{g}(p)}, & \text{otherwise.} \end{cases}
 \]

We claim that the family of subgroups $\set{H_1,\ldots,H_k}$ realizes $\chi$ in $G$. By construction, $H_i\leqslant G$ (since $\s{h_i}\succeq\s{g}$ for every $i\in[k]$ with $\chi(\set{i}) \neq 0$). Moreover, for every nonempty $I\subseteq[k]$, we have
\begin{equation} \label{eq: supp infty}
\supp_{-\infty}(H_I)
\,=\,
\bigcap_{i\in I}\supp_{-\infty}(H_i)
\,=\,
\set{p_J\st J\in\mathcal{M},\ I\subseteq J},
\end{equation} 
with the abuse of language that $\supp_{-\infty} (\set{0}) = \varnothing$.
Since $\supp(G)$ is finite, \Cref{eq: supp infty} holds, and 
$\chi$ is decreasing, it follows (again) that
\begin{align*}
\rk (H_I )=\infty 
&\ \Leftrightarrow\ 
\supp_{-\infty}(H_I)\neq\varnothing\\
&\ \Leftrightarrow\ 
I\subseteq J
\text{ for some }J\in\mathcal{M} \\
&\ \Leftrightarrow\ \chi(I)=\infty.
\end{align*}
Finally, the argument distinguishing between the finite cases $\set{0,1}$ is exactly the same as in the previous case.
\end{proof}

After identifying all the finite images into $\mathsf{F}$, conditions \ref{item: im rho} and $\ref{item: 0 minimal}$ are vacuous, and the corresponding binary configuration is obviously also decreasing. A characterization of binary intersection configurations follows immediately.

\begin{cor}
Let $G$ be a nontrivial subgroup of $(\QQ,+)$ and let $\chi \colon \Pk \to \set{\mathsf{F},\infty}$ be a binary intersection configuration. Then, $\chi$ is realizable in $G$ if and only if
 \begin{enumerate}[(i)]
\item $\chi$ is decreasing,
\item if $G$ has finite support then $|\mathcal{M}_{\infty}(\chi)|\leq |\supp_{-\infty}(G)|$. 
\end{enumerate} 
\end{cor}

\begin{proof}
The implication to the right follows from~\Cref{thm: main}. For the converse, let~$\chi$ be a binary intersection configuration satisfying~(i) and~(ii); replacing $\mathsf{F}$ to $1$, we obtain a rank intersection configuration satisfying all the hypotheses in~\Cref{thm: main} and hence, realizable in $G$.
\end{proof}




\section*{Acknowledgments}

The authors acknowledge support from the Spanish \emph{Agencia Estatal de Investigación} through grant PID2021-126851NB-I00 (AEI/FEDER, UE).

\renewcommand*{\bibfont}{\small}
\printbibliography

\Addresses

\end{document}